\documentclass[12pt]{article}
\usepackage[a4paper,top=2cm,bottom=2cm,left=1in,right=1in,marginparwidth=1.75cm]{geometry}
\usepackage[T1]{fontenc} 
\usepackage[utf8]{inputenc} 
\usepackage{lmodern} 
\usepackage{amsmath,amsthm,amsfonts,amssymb,mathtools,bbm}	
\usepackage{enumitem} 
\usepackage{microtype} 
\usepackage{mleftright} 
\usepackage{braket} 
\usepackage{esint} 
\mleftright

\usepackage{authblk}

\newtheorem{theorem}{Theorem}[section]
\newtheorem{proposition}[theorem]{Proposition}
\newtheorem{lemma}[theorem]{Lemma}
\newtheorem{corollary}[theorem]{Corollary}

\theoremstyle{definition}
\newtheorem{definition}[theorem]{Definition}

\theoremstyle{remark}
\newtheorem{remark}[theorem]{Remark}

\mathtoolsset{showonlyrefs}
\numberwithin{equation}{section}

\usepackage[hidelinks]{hyperref}

\def\XXint#1#2#3{{\setbox0=\hbox{$#1{#2#3}{\int}$}
\vcenter{\hbox{$#2#3$}}\kern-.5\wd0}}

\usepackage[breakable]{tcolorbox}
\usepackage{xcolor}
\usepackage{soul}
\soulregister\cite7
\soulregister\underline7
\soulregister\eqref7
\soulregister\ref7
\definecolor{commentcolor}{RGB}{230,230,230}
\definecolor{changecolor}{RGB}{230,230,200}
\definecolor{warningcolor}{RGB}{250,200,190}

\newtcolorbox{warningbox}{
    colback=warningcolor, 
    colframe=warningcolor, 
    boxrule=0pt,
    sharp corners, 
    left=0pt,
    right=0pt,
    top=0pt,
    bottom=0pt,
breakable}

\newtcolorbox{changebox}{
    colback=changecolor, 
    colframe=changecolor, 
    boxrule=0pt,
    sharp corners, 
    left=0pt,
    right=0pt,
    top=0pt,
    bottom=0pt,
breakable}

\newtcolorbox{commentbox}{
    colback=commentcolor, 
    colframe=commentcolor, 
    boxrule=0pt,
    sharp corners, 
    left=0pt,
    right=0pt,
    top=0pt,
    bottom=0pt,
breakable}

\DeclareMathOperator{\dist}{dist}
\newcommand{\dd}{\partial}
\newcommand{\abs}[1]{\lvert #1 \rvert}
\newcommand{\card}[1]{\operatorname{card}(#1)}
\newcommand{\R}{\mathbb{R}}
\newcommand{\N}{\mathbb{N}}

\newcommand{\D}{\nabla}
\newcommand{\A}{\mathbf{A}}
\newcommand{\spt}{\mathrm{spt}}
\renewcommand{\H}{\mathcal{H}}

\title{Good access to the crack and integrability of the full gradient for Griffith almost-minimizers in the plane}
\author{Camille Labourie, Lorenzo Lamberti, Antoine Lemenant\thanks{The three authors are affiliated to Université de Lorraine, CNRS, IECL, F-54000 Nancy, France.}}
\date{}

\begin{document}

\maketitle


\begin{abstract}
    We show that, for almost-minimizers of the Griffith energy in the plane, the complement of the crack is locally
    covered by a finite number of John domains where the displacement satisfies Hölder-type estimates. 
      As a consequence, we derive the   integrability of the full gradient and the finiteness of the traces along the crack. In particular, we show that any Griffith almost-minimizer in dimension two locally belongs  to the space $SBV^2$. 
\end{abstract}
 
\bigskip


\smallskip


\tableofcontents

\section{Introduction}
The Griffith energy provides a variational model for brittle fracture, where crack formation arises from the competition
between the elastic bulk energy and the surface energy of the crack. Given a bounded open set $\Omega \subset \R^N$
(the reference configuration of a linearly elastic body), the energy associated with a displacement field
$u \colon \Omega \setminus K \longrightarrow \R^N$ and a crack set $K \subset \Omega$ is given by
\begin{equation}
    \mathcal{G}(u,K)
    = \int_{\Omega \setminus K} \A e(u) : e(u)\,dx + \H^{N-1}(K),
\end{equation}
where $e(u)$ denotes the symmetric gradient and $\A$ is a positive definite elasticity tensor. In this model, the crack $K$ is
not prescribed a priori; its geometry instead is determined by the penalization of its $(N-1)$-dimensional Hausdorff measure.

A natural functional setting for this problem is the space $GSBD(\Omega)$ of generalized
special functions of bounded deformation, see \cite{dalmaso}, where the jump set represents the crack.  Within this
framework, compactness and lower semicontinuity results ensure the existence of minimizers of the Griffith
functional under general boundary conditions \cite{ChambolleCrismaleGSBD, ChambolleCrismaleGSBD2}. Furthermore,
minimizers are strong solutions in the sense that their jump set is essentially closed and the displacement
is smooth outside the crack \cite{CFI, CCI, ChambolleCrismaleStrong, FLK2}.

Beyond existence, the regularity of minimizers and their crack sets has become a major focus of recent research. An 
$\varepsilon$-regularity theory was pioneered in \cite{BabadjianIurlanoLemenant} for minimizers in the plane satisfying a
connectedness condition. Subsequent work refined the structure of the singular set \cite{LabourieLemenant},
and extended these results to arbitrary dimension under a separating condition \cite{LabourieLemenant2}. More recently,
an $\varepsilon$-regularity theorem was obtained for all almost-minimizers in the plane, yielding
$C^{1,\alpha}$ regularity at almost-every point on the crack \cite{FriedrichLabourie}.

In the present paper, we show uniform bounds for the geometric structure of the crack which, unlike 
$\varepsilon$-regularity, hold at all scales and locations. Our first main result states that each point outside the
crack has an ``escape path'' along which the distance to the crack increases proportionally to the arclength. This implies
that the components of $\Omega \setminus K$ cannot exhibit cusps, small acute angles, or narrow passes, and are
moreover locally finite. Our statement is as follows, and applies to topological almost-minimizers, a class that includes almost-minimizers, global minimizers, and minimizers as particular cases (see Section \ref{preliminaries} for the precise definitions).

\begin{proposition}\label{intro1}
    There exist two constants $\alpha_J\in(0,1)$ and $\tau>0$, depending only on $\A$, such that if $(u,K)$ is
    a topological almost-minimizer in $\Omega$ with gauge $h$, if $x_0\in\Omega\setminus K$ and $\ell > 0$ are such that
    \begin{equation}
        B(x_0,\ell) \subset \Omega \quad \text{and} \quad h(\ell) \leq \tau,
    \end{equation}
    then there exists a 1-Lipschitz curve $\gamma\colon[0,\ell]\longrightarrow\Omega\setminus K$ such that $\gamma(0)=x_0$ and
    \begin{equation*}
        \dist(\gamma(t),K)\geq \alpha_J t,\quad\forall t\in[0,\ell].
    \end{equation*}
\end{proposition}

The construction of escape paths originates from the theory of Mumford-Shah minimizers, as developed by Bonnet and David
\cite{bonnetDavid}. Following their approach, the first ingredient of Proposition~\ref{intro1}  is the uniform
rectifiability of the singular set, established in a recent result by the first author \cite{camilleUR}. The second
ingredient is the uniform concentration property, developed in another recent paper \cite{lablem}, which enables the
analysis of blow-up limits.

Building upon the existence of escape paths, we show that the crack can be locally accessed by John domains where the
displacement satisfies Hölder-type estimates.
We refer to Section \ref{sectionJohn} for the definition of a John domain, but we emphasize that this notion provides the appropriate setting for the validity of the Korn inequality.

\begin{theorem}\label{main0}
    There exists $\kappa \geq 3$, $\tau > 0$, $n_0 \in \N$ (which depend only on $\A$) such that the following holds.
    Let $(u,K)$ be a topological almost-minimizer for the Griffith energy with gauge $h$ in $\Omega$.
    Let $x_0 \in K$ and $r_0 > 0$ be such that
    \begin{equation*}
        B(x_0,\kappa r_0) \subset \Omega \quad \text{and} \quad h(\kappa r_0) \leq \tau.
    \end{equation*}
    Then   for all $r \in (0,r_0]$,
    there exists a finite covering $(U_i)_i$ of $B(x_0,r)\setminus K$ with at most $n_0$ connected open subsets of $B(x_0,3r) \setminus K$, i.e.  such that
    \begin{equation}
        B(x_0,r)\setminus K\subset \bigcup_i U_i\subset B(x_0,3r)\setminus K,
    \end{equation}
    and each $U_i$ is a John domain with center $x_i\in \partial B(x_0,2r)$ and with a John constant depending only on $\A$. Moreover, 
    \begin{equation*}
        |u(x)-u(y)|\leq Cr^{1/2},\quad \text{for all} \; x,y\in U_i,\quad \text{and for all} \; i,
    \end{equation*}
    where the constant $C>0$  may depend on $\bf A$, $r_0$ and on  $u$ itself, but does not depend on $r$.
\end{theorem}

We then apply our result to derive local $L^\infty$ and $H^1$ estimates for the displacement field.

\begin{theorem}\label{main1}
    Let $\Omega \subset \R^2$ be open and let $(u,K)$ be a topological almost-minimizer for the Griffith energy in $\Omega$. Then
    for any open set $U$ such that $\overline{U}\subset \Omega$ we have that $u\in L^{\infty}(U)\cap H^1 (U\setminus
    K)$. In particular, $u\in SBV^2(U)$.
\end{theorem}


This can be compared with Friedrich's theorem \cite[Theorem~2.1]{FriedrichPiecewise}, which states that in dimension
$N=2$,
any function in $GSBD^2$ lie in the space $SBV^p$ (for all $p < 2$)
after substracting a piecewise infinitesimal rigid motion.
The proof in \cite{FriedrichPiecewise} relies on the piecewise Korn inequality, which uses a general
decomposition result of a Lipschitz domain into open subsets that are John domains with uniform constants
\cite{friedrichJohn}. Because this decomposition is highly general, the resulting argument is necessarily technically
involved.

The key point of Theorem~\ref{main1} is that it applies specifically to almost-minimizing pairs $(u,K)$, enabling us to
derive regularity properties from minimality in a self-contained manner that is independent from \cite{FriedrichPiecewise,friedrichJohn}.
%

As a byproduct of our analysis, we also obtain a partial result toward the Griffith analogue of the main theorem in
\cite{labourieLemDegiorgi}, addressing a conjecture of De Giorgi. More precisely, we show that the displacement $u$
admits only finitely many traces as it approaches the singular set $K$.

\begin{corollary}\label{main3}
    There exists a constant $n_0\in \N$, depending only on $\A$, such that the following holds. Let $(u,K)$ be a topological almost-minimizer for
    the Griffith energy in $\Omega\subset \R^2$ and let $x_0\in K$. Let $E(x_0)$ be the set of values $a \in \R^2$ for which there
    exists a sequence $(y_k)_k$ in $\Omega \setminus K$ such that $y_k \to x_0$ and $u(y_k) \to a$.
    Then the cardinality of $E(x_0)$ is at most $n_0$.
\end{corollary}

The paper is organized as follows. In Section~3, we establish the existence of escape paths for Griffith almost-minimizers. Section~4 is devoted to Morrey--Campanato type estimates and their applications; in particular, Theorem~\ref{main0} is proved in Section~4.3 and Corollary~\ref{main3} in Section~4.4. Finally, Section~5 contains the proof of Theorem~\ref{main1}.

\section{Preliminaries and notation}\label{preliminaries}

Although the following definitions and preliminaries hold in any dimension $N \geq 2$, we present them in $\R^2$ to
align with the two-dimensional setting, where the main results are obtained. 

\subsection{Linear elasticity setting and notations}

In the sequel, $\Omega \subset \R^2$ denotes a fixed open set.
We regard $\Omega$ as the reference configuration of a homogeneous, isotropic, linearly elastic body.
We fix a fourth-order elastic tensor $\A$ such that there exists a constant $c > 0$ with
\begin{equation}
    \A(\xi - \xi^T) = 0 \quad \text{and} \quad \A \xi : \xi \geq c \abs{\xi + \xi^T}^2
    \quad \text{for all $\xi \in \R^{2 \times 2}$}.
\end{equation}
A standard example is the Lamé tensor, defined by
\begin{equation}
    \A \xi=\lambda ({\rm tr}\xi) I+2 \mu \xi, \quad \forall\xi \in \mathbb M^{2 \times 2}_{\rm sym},
\end{equation}
where $\lambda$ and $\mu$ are the Lam\'e coefficients satisfying $\mu>0$ and $\lambda+\mu>0$. 
We denote by
\begin{equation}
    e(u)=\frac{\D u+\D u^T}{2}
\end{equation}
the \emph{symmetric gradient} of $u$. We call \emph{rigid motion} affine maps $a \colon \R^2 \longrightarrow \R^2$ of
the form $a(x) = Ax+b$, where $b \in \R^2$ and $A \in \R^{2 \times 2}$ is a skew-symmetric matrix.

We follow the convention of letting $\N$ denote the set of nonnegative integers $\{0, 1, 2, ...\}$.
For $x_0\in\R^2$ and $r>0$, we let $B(x_0,r)$ represent the open disk with center $x_0$ and radius $r>0$.
If the disk is centered at the origine, we simply write $B_r$.
For $M \geq 0$, the notation $M B(x_0,r)$ refers to the disk $B(x_0,M r)$. 
Given a curve $\gamma : I \longrightarrow \R^2$ defined on an interval $I$, we denote by $\spt(\gamma)$ the set $\{\gamma(t) \mid t \in I\}$.
We use the letter $C$ to denote a generic constant $\geq 1$ which depends on $\A$ and may possibly vary from line to line.

\subsection{Competitors}

\begin{definition}\strut
    \begin{enumerate}[wide,label=(\alph*)]
        \item An \emph{admissible pair} $(u,K)$ consists of a relatively closed subset $K\subset\Omega$ and a vector
            field $u \in H^1_{\mathrm{loc}}(\Omega \setminus K)$.

        \item A relatively closed subset $K \subset \Omega$ is called \emph{coral} if
            \begin{equation}
                \H^1(K \cap B(x,r))>0, \quad\forall x \in K,\,\forall r>0.
            \end{equation}
            We also say that a pair $(u,K)$ is \emph{coral} when $K$ is coral.

        \item Let $(u,K)$ be an admissible pair.
            Let $x \in \Omega$ and $r > 0$ be such that $\overline{B}(x,r) \subset \Omega$.
            A \emph{competitor} of $(u,K)$ in $B(x,r)$ is an admissible pair $(v,L)$ such that
            \begin{equation}
                L \setminus B(x,r) = K \setminus B(x,r) \quad \text{and} \quad v = u \quad \text{a.e. in } \Omega \setminus \left(K \cup B(x,r)\right).
            \end{equation}
            A competitor $(v,L)$ of $(u,K)$ in $B(x,r)$ is called \emph{topological} if $L \setminus B(x,r) = K \setminus B(x,r)$ and
            \begin{gather}\label{eq_topo_competitor}
                \text{all points $x,y \in \Omega \setminus (K \cup B)$ that are separated by $K$}\\
                \text{are also separated by $L$}.
            \end{gather}
            In other words, if $x,y \in \Omega \setminus (K \cup B)$ belong to different connected components of $\Omega \setminus K$, they also belong to different connected components of $\Omega \setminus L$.
    \end{enumerate}
\end{definition}

\subsection{Local minimizers}

A \emph{gauge} is a non-decreasing function $h \colon (0,+\infty) \longrightarrow [0,+\infty]$ satisfying $\lim_{r\to 0} h(r)=0$.
\begin{definition}[Almost-minimizers]\label{defAlmostMin}
    Let $h$ be a gauge and $\Omega \subset \R^2$ be open. An \emph{almost-minimizer} of the Griffith functional with
    gauge $h$ in $\Omega$ is a coral pair $(u,K)$ such that for all $x \in \Omega$ and $r > 0$ with $\overline{B}(x,r)
    \subset \Omega$ and for any competitor $(v,L)$ of $(u,K)$ in $B(x,r)$, we have
    \begin{multline*}
        \int_{B(x,r) \setminus K} \A e(u):e(u) \,dx + \H^1(K \cap B(x,r))
        \\\leq \int_{B(x,r) \setminus L} \A e(v) : e(v) \,dx + \H^1(L \cap B(x,r)) + h(r) r.
    \end{multline*}
    A \emph{topological almost-minimizer} of the Griffith functional with gauge $h$ in $\Omega$ is a pair which satisfies the above definition but only with respect to topological competitors. 
\end{definition} 

We will also consider the notion of \emph{global minimizer}, which arises as blow-up limits of almost minimizers as was proved in \cite{lablem}. Here is the definition.

\begin{definition}[Global minimizers]
    A \emph{global minimizer} is a coral pair $(u,K)$ such that for all $x \in K$, $r > 0$ and for any topological
    competitor $(v,L)$ of $(u,K)$ in $B(x,r)$, we have
    \begin{equation}
        \int_{B(x,r)} \A e(u):e(u) \,dx+ \H^1(K \cap {B(x,r)})
        \leq \int_{B(x,r)} \A e(v) : e(v) \, dx + \H^1(L \cap B(x,r)).
    \end{equation}
\end{definition} 

\begin{remark}Notice that all our results are given on \emph{topological} almost-minimizers, which defines a larger class than that of almost-minimizers, since the admissible competitors are more restricted.  Moreover, since a global minimizer is a topological minimizer, the notion of  topological almost-minimizers provides a convenient framework that includes both almost-minimizers and global minimizers. 
\end{remark}

\begin{remark}[Standard rescaling of almost-minimizers]\label{rmk_scaling}
    If $(u,K)$ is an almost-minimizer with gauge $h$ in a   ball $B(x_0,r_0)$, then the pair $(u_0,K_0)$ defined by
    \begin{equation}
        u_0(x) := r_0^{-1/2} u(x_0 + r_0 x) \quad \text{and} \quad K_0 := r_0^{-1}(K - x_0),
    \end{equation}
    is an almost-minimizer with gauge $\tilde{h}$ in $B_1$, where $\tilde{h} : t \mapsto h(r_0 t)$.
\end{remark}

We recall the following result, which was proved, actually, in dimension $N\geq 2$.

\begin{theorem}[Ahlfors-regularity]\label{ARTheorem}
    There exist two constants $\varepsilon_0 \in (0,1)$ and $C_A \geq 1$ (which depends only on $\A$) such that the following holds.
    Let $(u,K)$ be a topological almost-minimizer in $\Omega$ with gauge $h$.
    Then for all $x \in \Omega \cap K$ and $r > 0$ such that $B(x,r) \subset \Omega$ and $h(r) \leq \varepsilon_0$, we have
    \begin{equation}\label{eq_AR1}
        \H^1(K \cap B(x,r)) \geq C_A^{-1} r.
    \end{equation}
\end{theorem}
For details, we refer to \cite{FLK2} which generalizes the method of \cite{CCI}, \cite{CFI} to topological quasiminimizers.
Up to choosing $C$ a bit larger (still depending only on $\A$), it is standard that for all $x \in \Omega$ and $r > 0$
such that $B(x,r) \subset \Omega$, we also have the upper bound
\begin{equation}\label{eq_AR2}
    \int_{B(x,r)\setminus K} |e(u)|^2 \, dx + \H^1(K \cap B(x,r)) \leq C(1 + h(r)) r.
\end{equation}

\subsection{Normalized quantities}

We introduce a notation that will be used throughout the paper. Let $(u,K)$ be a Griffith almost-minimizer. For $x_0 \in K$
and $r_0 > 0$ such that $B(x_0,r_0) \subset \Omega$, \emph{the bilateral flatness} $\beta(x_0,r_0)$ of $K$ in
$B(x_0,r_0)$ is defined as
\begin{equation}
    \beta_K(x_0,r_0) := r^{-1} \inf_{P} \max\biggl(\sup_{x \in K \cap B(x_0,r_0)} \dist(x,P), \sup_{x \in P \cap
    B(x_0,r_0)} \dist(x,K)\biggr),
\end{equation}
where $P$ ranges among affine hyperplanes passing through $x_0$ (the infimum is always attained by compactness of the
Grassmann space). 
To make the notation lighter, we drop the dependence on $K$ in the definition of the flatness.

We also define the normalized energy in the ball $B(x_0,r_0)$ as
\begin{equation}
    \omega(x_0,r_0)=r^{-1} \int_{B(x_0,r_0)\setminus K} |e(u)|^2 \;dx,
\end{equation}
and the $p$-normalized elastic energy as
\begin{equation}
    \omega_p(x_0,r_0)=r^{1-\frac{4}{p}}\biggl(\int_{B(x_0,r_0)\setminus K} |e(u)|^p \;dx\biggr)^\frac{2}{p}.
\end{equation}

Both quantities $\beta(x_0,r_0)$ and $\omega_p(x_0,r_0)$ are invariant under the standard rescaling given in
Remark~\ref{rmk_scaling}. Notice that for $t \in (0,1)$, one trivially has the scaling inequalities
\begin{equation}\label{eq_scaling}
    \beta(x_0,tr_0) \leq t^{-1} \beta(x_0,r_0) \quad \text{and} \quad \omega_p(x_0,t r_0) \leq t^{1-4/p} \omega_p(x_0,r_0).
\end{equation}
We will see in Lemma~\ref{LemmaCarleson} that $\beta$ and $\omega_p$ are often small, as a consequence of
energy estimates and uniform rectifiability of the crack set.

\subsection{John Domains} \label{sectionJohn}

In this subsection, we additionally assume that the open set $\Omega\subset\R^2$ is bounded.

\begin{definition}[John domains]
    We say that a bounded open set $\Omega \subset \R^2$ is a John domain with center $x_0 \in \Omega$ if there exists
    a constant $\alpha_J \in (0,1)$ such that for every $x\in\Omega$ there exists a $1$-Lipschitz 
    curve $\gamma \colon [0,\ell] \longrightarrow \Omega$ such that $\gamma(0)=x$, $\gamma(\ell)=x_0$ and 
    \begin{equation}\label{JohnCondition}
        \dist(\gamma(t),\Omega^c)\geq \alpha_J t,\quad \forall t \in [0,\ell].
    \end{equation}
    The curve $\gamma$ is called a  John curve from $x$ to $x_0$.
\end{definition}
John domains may have corners, provided that the angles are not too small, but they cannot have outward cusps. Typical
examples of John domains include Lipschitz and snowflake domains. Notice that the above definition is scale invariant
(the John constant $\alpha_J$ is invariant under homotheties of the domain).

The arc-length reparametrization of a $1$-Lipschitz curve satisfying \eqref{JohnCondition} still satisfies
\eqref{JohnCondition} with the same constant $\alpha_J$. Thus one may always assume for free that a John curve is arc-length parametrized.

\begin{remark}\label{boundl}
    Let $\gamma : [0,\ell] \longrightarrow \Gamma$ be a John curve connecting a point $x$
    to a point $x_0$. As $\gamma$ is $1$-Lipschitz, one has $\abs{x - x_0} \leq \ell$ and therefore
    \begin{equation*}
        \dist(x_0,\Omega^c) = \dist(\gamma(\ell),\Omega^c) \geq \alpha_J \ell \geq \alpha_J \abs{x - x_0}.
    \end{equation*}
    One deduces that the center $x_0$ of a John domain must always be relatively far
    from the boundary:
    \begin{equation}
        \dist(x_0,\Omega^c) \geq \alpha_J \sup_{x \in \Omega} \abs{x - x_0}
        \geq \Bigl(\frac{\alpha_J}{2}\Bigr) \mathrm{diam}(\Omega).
    \end{equation}
    One can see similarly that the length $\ell$ of a John curve is also always bounded from above by the diameter of
    the domain. Indeed,
    \begin{equation}
        \alpha_J \ell \leq \dist(x_0,\Omega^c)\leq {\rm diam}(\Omega),
    \end{equation}
    so $\ell \leq {\rm diam}(\Omega)\alpha_J^{-1}$. 
\end{remark}


\begin{remark}\label{boundaryJ}
    If $\Omega$ is a John domain with center $x_0$, then there exists a John curve starting not only from every point
    of $\Omega$ but also from every point of $\dd\Omega$. Indeed, let $x\in\dd\Omega$. There exists a sequence
    $(x_n)_{n\in\N}\subset\Omega$ such that $x_n \to x$. For each $n$, let $\gamma_n : [0,\ell_n] \longrightarrow
    \Omega$ be a John curve joining $x_n$ to $x_0$. By Remark~\ref{boundl}, and extending each $\gamma_n$ constantly
    on $[\ell_n,{\rm diam}(\Omega)\alpha_J^{-1}]$, we may regard each $\gamma_n$ as defined on the same segment
    $I:=[0,{\rm diam}(\Omega)\alpha_J^{-1}]$. Since the curves $\gamma_n$ are all $1$-Lipschitz, Arzel\`a-Ascoli's
    theorem yields that, up to a subsequence, $(\gamma_n)_n$ converges uniformly to a $1$-Lipschitz curve $\gamma$ on $I$.
    We can also assume that $\ell_n\to \ell \in I$ by compactness of $I$. Then, by uniform convergence we have
    $\gamma(0)=x$ and $\gamma(\ell)=x_0$. Writing \eqref{JohnCondition} for each $\gamma_n$ and passing to the limit,
    we also see that $\gamma$ satisfies \eqref{JohnCondition}.
\end{remark}

Later on we will need a slightly more general definition of \emph{escape paths}.

\begin{definition}[Escape path] 
    Let $G\subset \R^2$ be a set.
    We say that  a  $1$-Lipschitz curve
    $\gamma:[0,\ell] \longrightarrow \R^2 \setminus G$ is an escape path from $x:=\gamma(0)$ to $y:=\gamma(\ell)$ with constant $\alpha_J \in (0,1)$ if it satisfies
    \begin{equation}\label{JohnConditionK}
        \dist(\gamma(t),G)\geq \alpha_J t,\quad \forall t \in [0,\ell].
    \end{equation}
\end{definition}

\begin{remark} If $\Omega$ is a John domain then any John curve is an escape path with $G=\Omega^c$. Also, later on we will prove the existence of escape paths in the context of Griffith minimizers, for which in this case the definition will apply with $G=K$. 
\end{remark}


The following technical tool will be needed later. It says that from an escape path one can construct a chain of balls
with nice summability properties.

\begin{proposition}
    \label{Chain}
    Let $\Omega\subset\R^2$ be an open set, $x,x_0\in\Omega$ and let $\gamma \colon [0,\ell] \longrightarrow \Omega$ be
    a John curve from $x$ to $x_0$ satisfying \eqref{JohnCondition}. Then there exist $M>1$, $q \in (0,1)$ (which only
    depend on $\alpha_J$) and a sequence of balls $(D_i)_{i\in\N} \subset\Omega$ such that the following properties hold:
    \begin{enumerate}[label=\roman*)]
        \item\label{Chain1} $D_0=B(x_0,\alpha_J \ell/4)$;
        \item\label{Chain2} for all $i \in \N$, $\mathrm{diam}(D_{i+1})=q \mathrm{diam}(D_i)$ and
            \begin{equation}\label{eqeq13}
                |x_i-x|\leq q^i\ell;
            \end{equation}
        \item\label{Chain4} for every $i\in\N$ there exists a ball  $V_i\subset D_i\cap D_{i+1}$ such that $D_i\cup D_{i+1}\subset MV_i \subset \Omega$. 
    \end{enumerate}
\end{proposition}
\begin{proof}
 
    Setting
    \begin{equation}
        q:=\frac{8-\alpha_J}{8+\alpha_J}\in(0,1),
    \end{equation}
    we define for $i\in\N$,
    \begin{equation}
        t_i:=q^i\ell,\quad x_i:=\gamma(t_i),\quad r_i=\frac{\alpha_J}{4} t_i=\frac{\alpha_J}{4} q^i\ell.
    \end{equation}
    We further set $D_i:=B(x_i,r_i)$. Fix $i\in\N$. Since $\gamma$ is an escape path, we have
    \begin{equation}
        \dist(x_i,\Omega^c)\geq \alpha_J t_i=4r_i> r_i, \label{boundInteressant}
    \end{equation}
    which implies $D_i\subset\Omega$.
    Property \ref{Chain1} and the first part of \ref{Chain2} are clearly satisfied.
    By $1$-Lipschitz continuity of $\gamma$, we have for all $i, j$,
    \begin{equation}
        |x_i-x_j|=|\gamma(t_i)-\gamma(t_j)|\leq |t_i - t_j|,
    \end{equation}
    and in particular \eqref{eqeq13} by letting $j \to +\infty$.
    Taking $j = i+1$ and using our choice of $q$, this also yields
    \begin{equation}
        |x_i-x_{i+1}|\leq (1-q)t_i=\frac{\alpha_J}{8} (1 + q) t_i= \frac{1}{2}(r_i + r_{i+1}),
    \end{equation}
    and thus
    \begin{equation}
        \overline{B}(x_i,r_i/2)\cap \overline{B}(x_{i+1},r_{i+1}{/2})\neq \emptyset.
    \end{equation}
    Now, let $z_i$ be any point in this intersection and set $\rho_i = r_{i+1}/2$.
    Since $r_{i+1} = q r_i$ and $q \in (0,1)$, it is straightforward that
    \begin{equation}
        B(z_i,\rho_i) \subset D_i \cap D_{i+1} \quad \text{and} \quad D_i \cup D_{i+1} \subset B(z_i,2r_i) = M
        B(z_i,\rho_i),
    \end{equation}
    where $M = 4/q$   and $V_i=B(z_i,\rho_i)$. It remains to prove that $MV_i\subset \Omega$. Notice that  the radius of $MV_i$ is equal to $2r_i$ and its center is $z_i\in B(x_i,r_i/2)$.  Therefore, $MV_i\subset B(x_i,4r_i)$.  On the other hand we know that ${\rm dist}(x_i, \Omega^c)\geq \alpha_Jt_i=4r_i$, so $MV_i\subset B(x_i,4r_i)\subset \Omega$ and so follows \ref{Chain4}.
    \end{proof}


We conclude this section on John domains with the following technical lemma. It shows that adding a small segment to the starting point or the endpoint of an escape path preserves the escape-path property, possibly after a slight modification of the constants.

\begin{lemma}[Adding a segment to the endpoint or the starting point]
    \label{LemmaJohnSegment}
    Let $G\subset \R^2$ be a set and let $\gamma \colon [0,\ell] \longrightarrow \R^2 \setminus G$ be an escape path from  $\gamma(0) = x$ to $\gamma(\ell)=x_0$ with constant $\alpha_J \in (0,1)$. Then
    \begin{enumerate}[label=(\arabic*)]
        \item for every  $y \in B(x_0, \alpha_J \ell/2)$, there exists an escape path from $x$ to $y$ with constant
            $\alpha_J/4$;
        \item for every $y\in B(x, \dist(x,G))$, there exists an escape path from $y$ to $x_0$ with constant $\alpha_J/4$. 
    \end{enumerate}
\end{lemma}

\begin{proof}  We start with the proof of the first assertion. Let $y \in B(x_0, \alpha_J \ell/2)$ and define $\tilde{\gamma}$ by
    \begin{align*}
        \tilde{\gamma}(t)=
        \begin{cases}
            \gamma(t)                                        &\forall t\in[0,\ell],\\
            x_0 + \bigl(\frac{t - \ell}{\ell}\bigr) (y-x_0)  &\forall t\in[\ell,2\ell].
        \end{cases}
    \end{align*}
    The curve $\tilde{\gamma}$ is $1$-Lipschitz because $\abs{y - x_0} \leq \alpha_J \ell / 2 < \ell$.
    For $t\in[0,\ell]$, it is immediate that 
    \begin{equation} 
        \dist(\gamma(t),G)\geq \alpha_J t \geq \frac{\alpha_J}{4} t.
    \end{equation}
    And for $t\in [\ell,2\ell]$, we bound $\dist(x_0,G) \geq \alpha_J \ell$ and $\abs{y - x_0} \leq \alpha_J \ell/2$ to obtain
    \begin{align*}
        \dist(\tilde{\gamma}(t),G)\geq \dist(x_0,G) - \Bigl(\frac{t-\ell}{\ell}\Bigr)|y-x_0|\geq
        \frac{\alpha_J}{2}(3\ell - t) \geq \frac{\alpha_J}{4}t,
    \end{align*}
    which is the thesis. 

    \medskip

    We now prove the second assertion. Let $y\in B(x,\mathrm{dist}(x,G))$. We will obtain the new escape path as  the
    union of the segment $[y,x]$ with $\mathrm{spt}(\gamma)$. Before that, we first justify that 
    \begin{equation}
        \label{eqeq41}
        \mathrm{dist}(\gamma(s),G)\geq\frac{\alpha_J}{4}(s+\mathrm{dist}(x,G)),\quad\forall s\in[0,\ell].
    \end{equation}
    We distinguish two cases. If $s\geq \mathrm{dist}(x,G)/2$, then 
    \begin{equation}
        \mathrm{dist}(\gamma(s),G)\geq\alpha_Js \geq \frac{\alpha_J}{4}(s+\mathrm{dist}(x,G)).
    \end{equation}
    If $s \leq \mathrm{dist}(x,G)/2$, we use the $1$-Lipschitzness of $\gamma$ to bound from below
    \begin{equation}
        \mathrm{dist}(\gamma(s),G)
        \geq \mathrm{dist}(x,G)-|\gamma(s)-x|
        \geq \mathrm{dist}(x,G)-s \geq \frac{\mathrm{dist}(x,G)}{2}
    \end{equation}
    and then
    \begin{equation}
        \mathrm{dist}(\gamma(s),G)
        \geq \frac{\alpha_J}{4}(s+\mathrm{dist}(x,G)).
    \end{equation}
    The estimate \eqref{eqeq41} follows from the previous two cases.

    We now define $\tilde{\gamma}\colon[0,\ell+t_0] \longrightarrow \R^2$ as follows:
    \begin{equation}
        \tilde{\gamma}(t)=
        \begin{cases}
            y+\frac{t}{t_0}(x-y),\quad&\text{for }t\in[0,t_0],\\
            \gamma(t-t_0) &\text{for }t\in[t_0,\ell+t_0].
        \end{cases}
    \end{equation}
    where $t_0=|x-y|$.
    It is obvious that $\tilde{\gamma}$ is 1-Lipschitz. It is left to check the escape condition. If $t\in[0,t_0]$, then using 
    $|x-y|<\mathrm{dist}(x,G)$, we have
    \begin{equation}
        \mathrm{dist}(\tilde{\gamma}(t),G)\geq \mathrm{dist}(x,G)-|\tilde\gamma(t)-x|=\mathrm{dist}(x,G)-(|x-y|-t)\geq t.
    \end{equation}
    If $t\in[t_0,t_0+\ell]$, then by \eqref{eqeq41} we estimate
    \begin{equation}
        \mathrm{dist}(\tilde{\gamma}(t),G)=\mathrm{dist}(\gamma(t-t_0),G)\geq \frac{\alpha_J}{4}(t-t_0+\mathrm{dist}(x,G))\geq \frac{\alpha_J}{4}t.
    \end{equation}
    Collecting the previous two inequalities we get the thesis.
\end{proof}

\subsection{Korn Inequalities}

We recall two well-known Korn inequalities, which one can find for instance in \cite{korn}. The first one gives a
control on $u$ by the symmetric gradient, up to a rigid motion $a$: 
\begin{equation}
    \int_{B(x_0,r)} |u-a| \,dx \leq C r \int_{B(x_0,r)} |e(u)| \,dx. \label{Korn1}
\end{equation}
The rigid motion $a$ above is explicit. Actually, one can take any continuous operator $u\mapsto a$ which is
the identity on rigid motions, as for instance
\begin{equation}\label{eq_explicit_Korn}
    a(u)(x)=\left(\frac{1}{|\omega|}\int_{\omega} u \,dx\right) + \left(\frac{1}{|\omega'|}\int_{\omega'} \frac{\nabla u - \nabla^T u}{2} \,dx\right)(x-x_0),
\end{equation}
where 
\[x_0:=\frac{1}{|\omega|}\int_{\omega} x \,dx,\]
and $\omega,\omega'$ are any two non-negligible subsets of $B(x,r)$; the constant $C$ then depends on the choice of
$\omega$ and $\omega'$.
Notice that $x_0$ does not depend on $\omega'$ but only on $\omega$.

The second Korn inequality that can also be found in \cite{korn} is a control of the full gradient, up to a skew-symmetric matrix, in terms of its symmetric part alone: 
\begin{equation}
    \int_{B(x,r)} |\nabla u-A(u)| \,dx \leq C \int_{B(x,r)} |e(u)| \,dx. \label{Korn2}
\end{equation}
Here again, the skew-symmetric matrix $A$ is explicit and can be for instance 
\[A(u):=\left(\frac{1}{|\omega|}\int_{\omega} \frac{\nabla u - \nabla^T u}{2} \,dx\right),\]
where $\omega$ is any non-negligible subset of $B(x,r)$, and $C$ depends on the choice of $\omega$.
However, any linear operator that acts as the identity on skew-symmetric matrices would work.

For any ball $B(x_0,r)$, we set
\begin{equation}
    A_{B(x_0,r)}=\fint_{B(x_0,r)}\frac{\D u-\D u^T}{2}\,dx.
\end{equation}
The matrix $A_{B(x_0,r)}$ is sometimes called \emph{infinitesimal rotation associated to $u$ in $B(x_0,r)$.}
More generally, Korn inequalities remain valid in a John domain, with a constant $C$ depending additionally on the John-domain constant.
Here is a statement that follows for instance by gathering \cite[Theorem 1.5]{korn1} and \cite[Theorem 8]{korn2}.

\begin{proposition}\label{kornJohn}
    Let $\Omega \subset \R^2$ be a  John domain with constant $\alpha_J\in (0,1)$. Then for all $u\in W^{1,p}_{loc}(\Omega;\R^2)$
    and $p\in (1,+\infty)$ there exists a rigid motion $a(x):=Ax+b$ such that 
    \[\int_{\Omega} |u-a|^p\,dx \leq C {\rm diam}(\Omega)^p \int_{\Omega} |e(u)|^p \,dx,\]
    \[\int_{\Omega} |\nabla u -A|^p\,dx \leq C \int_{\Omega} |e(u)|^p \,dx,\]
    where $C \geq 1$ is a constant that depends only on $\alpha_J$ and $p$. 
\end{proposition}

\section{Escape paths}

\subsection{No bounded connected component in the complement of a global minimizer}

In this section, we prove that $\R^2\setminus K$ has no bounded connected components when $K$ is the singular set of
a global minimizer. It is worth noting that the result extends to $\R^N$ ($N\geq 2$) with the same proof. 


\begin{proposition}[No isolated component]\label{noiso}
    Let $(u,K)$ be a global minimizer of the Griffith functional. Then $\R^2\setminus K$ has no bounded connected component.
\end{proposition}

\begin{proof}
    Let us assume by contradiction that $\R^2\setminus K$ has a bounded connected component $\Omega_0$. By minimality, $u$
    must be a rigid motion in $\Omega_0$ (otherwise, the replacement of $u$ by a rigid motion inside $\Omega_0$ costs less energy).  Moreover, adding a rigid motion to a global minimizer still yields a global minimizer, so we may
    assume that $u(x)=a_+(x)$ in $\Omega_0$, where $a_+$ is an arbitrary rigid motion $a_+$ is arbitrary which will be chosen later.  It
    is standard by uniform density estimates \eqref{eq_AR1}, \eqref{eq_AR2} and the rectifiability of $K$
    \cite{camilleUR}, that for $\H^1$-almost every $x_0\in K$ the following two properties hold true: 
    \begin{equation}\label{Pr1}
        \lim_{r\to 0^+}\frac{1}{r}\int_{B(x_0,r)}|e(u)|^2\,dx=0
        \qquad\text{and}\qquad
        \lim_{r\to 0^+} \beta(x_0,r)=0.
    \end{equation}
    Let $x_0\in\dd^*\Omega_0\subset K$ satisfy \eqref{Pr1}. Let us fix $\delta\in(0,1/10)$ and $\varepsilon>0$ that
    will be chosen later. There exists $r>0$ such that 
    \begin{equation}
        \label{eqeq25}
        \int_{B(x_0,r)}|e(u)|^2\,dx\leq \varepsilon r
    \end{equation}
    and such that for some line $P$ passing through $x_0$,
    \begin{equation}
        K \cap \overline{B(x_0,r)} \subset \set{x \in \overline{B(x_0,r)} | \mathrm{dist}(x,P) \leq \delta r}. 
    \end{equation}
    Without loss of generality (see Remark~\ref{rmk_scaling}), we may assume that $x_0 = 0$.
    We also assume that $P = \{x_2 = 0\}$ to simplify the notation. 

    We denote 
    \begin{equation}
        D_+:=\set{x\in B_r | x_2 > \delta r}
        \quad \text{and} \quad
        D_-:=\set{x\in B_r | x_2 < -\delta r}.
    \end{equation}
    Furthermore, since $0 \in \partial^*\Omega_0$, we may assume that for $r$ small enough, $|B_r\cap \Omega_0|>
    |B_r \setminus (D_+ \cup D_-)|$. Thus, $B_r \cap \Omega_0$ cannot be contained in $B_r \setminus (D_+ \cup
    D_-)$, and $\Omega_0$ must therefore contain $D_+$ or $D_-$ by connectedness. In the sequel we assume without
    loss of generality that $D_+ \subset \Omega_0$; in other words, we have $u(x) = a_+(x)$ in $B_r\cap D_+$.

    We define the annulus $R:=B_r\setminus B_{3r/4}$. By Korn's inequality, there exists a universal constant $C>0$
    independent of $\delta$ and a rigid motion $a_-$ such that
    \begin{equation}\label{eq_Korn}
        \int_{R\cap D_-} |u-a_-|^2 \leq C r^2\int_{B_r}|e(u)|^2\,dx.
    \end{equation}
    Then we define 
    \begin{equation}
        w = \varphi u + (1-\varphi) a_+,
    \end{equation}
    where $\varphi$ is a cut-off function such that $\varphi=1$ on $B_r \cap \{x_2 \leq -3\delta r\}$, $\varphi=0$ on
    $B_r \cap \{x_2 \geq -2\delta r\}$ and $|\nabla \varphi|\leq C/(\delta r)$.  Notice that $w=u$ on $\partial B_r
    \setminus \{|x_2| \leq 3\delta r\}$ since $u = a_+$ on $D_+ = B_r \cap \set{x_2 > \delta r}$. A computation
    shows that 
    \begin{equation}\label{eq_ew}
        e(w) = \varphi e(u)+\nabla \varphi \odot (u - a_+).
    \end{equation}
    Observing that we only have $\nabla \varphi \ne 0$ on $B_r \cap \{-3\delta r \leq x_2 \leq -2\delta r\} \subset D_-$, we estimate 
    \begin{equation}
        \int_{B_r} \abs{e(w)}^2 \, dx \leq C \int_{B_r} \abs{e(u)}^2 \, dx + C (\delta r)^{-2} \int_{B_r \cap D_-} \abs{u - a_+}^2 \, dx
    \end{equation}
    Choosing $a_+$ equal to $a_-$, Korn inequality \eqref{eq_Korn} and \eqref{eqeq25} imply that
    \begin{equation}
        \label{eqeq26}
        \int_{B_r} \abs{e(w)}^2 \, dx \leq C(\delta) \int_{B_r} \abs{e(u)}^2 \, dx \leq C(\delta) \varepsilon r,
    \end{equation}
    where $C(\delta)$ also depends on $\delta$.

    Now we are in position to define a competitor $(u',K')$ of $(u,K)$ in $B_r$ by setting
    \begin{equation}
        K'=(K\setminus B_r)\cup Z,
    \end{equation}
    where $Z = \{x \in \dd B_r \;:\; |x_2|\leq 3 \delta r \}$, and
    \begin{equation}
        u'=
        \begin{cases}
            w & \text{in} \; B_r,\\
            u & \text{in} \; \Omega \setminus (B_r\cup Z).
        \end{cases}
    \end{equation}
    As $w$ and $u$ glue well along $\dd B_r\setminus Z$, the vector field $u'$ has locally weak derivatives in
    $\Omega \setminus K'$. Furthermore, $(u',K')$ is clearly a topological competitor for $(u,K)$ in any given ball $B_0$
    which sufficiently large that $B(x,2r) \cap \Omega_0 \subset B_0$.
    By minimality, we have
    \begin{equation}
        \int_{B_r}\A e(u) : e(u)\,dx+\H^1(K\cap\overline{B_r}) \leq \int_{B_r}\A e(u') : e(u')\,dx+\H^1(Z).
    \end{equation}
    Recalling that $C_A$ is the Ahlfors-regularity constant appearing in Theorem~\ref{ARTheorem} and \eqref{eqeq26}, we infer
    \begin{equation}
        C_A^{-1} r \leq C(\delta) \varepsilon r + C \delta r.
    \end{equation}
    Choosing $\delta<(4C_A C)^{-1}$ and $\varepsilon<(4C_A C(\delta))^{-1}$, we finally get a contradiction.
    %
\end{proof}

\begin{remark} 
    It would be possible to obtain an analogue of Proposition \ref{noiso} for Griffith almost-minimizers,
    rather than only for global minimizers. See for instance \cite[Proposition 68.1]{d} for a statement in the scalar case.
\end{remark}


\subsection{Construction of escape paths}

We now prove that from any point outside the crack has an escape curve. To this end, we need to develop some preliminary
results. The first of these is the following standard density estimate, similar to Proposition 4.1. in
\cite{BabadjianIurlanoLemenant}.

\begin{lemma}
    \label{LemmaDensity}
    For every $p\in(1,2]$, there exists a constant $C=C(p,\A) \geq 1$ such that the following holds: if $(u,K)$ is a
    global minimizer of the Griffith functional, for any $\varepsilon>0$, $x_0\in K$ and $r>0$ such that
    \begin{equation}
        \omega_p(x_0,r)+\beta(x_0,r) \leq \varepsilon,
    \end{equation}
    there exists $\rho \in (r/2,r)$ such that
    \begin{align*}
        \H^1(K\cap \overline{B(x_0,\rho)})-2\rho\leq C\varepsilon \rho. 
    \end{align*}
\end{lemma}

\begin{proof}
    Let $x_0 \in K$ and $r > 0$ be such that $\omega_p(x_0,r)+\beta(x_0,r)\leq\varepsilon$. Without loss of generality,
    we may assume that $B(x_0,r) = B_1$ (see Remark~\ref{rmk_scaling}).  We let $P$ denote a line passing through
    the origin such that
    \begin{equation}
        K \cap B_1 \subset \{x \in B_1 \, : \, \mathrm{dist}(x,P) \leq \beta(0,1)\}.
    \end{equation}
    Without loss of generality, we may assume that $P=\{x_2=0\}$.
    Let $a_\pm$ be a rigid motion, with matrix $A_{\pm}$, associated to $D_\pm = B_1 \cap \{\pm x_2>\varepsilon\}$.
    We thus have by Korn inequality,
    \begin{equation}\label{eq_r1}
        \int_{D_{\pm}} \abs{\D u - A_\pm}^p \, dx \leq C \int_{D_{\pm}} \abs{e(u)}^p \, dx \leq C \omega_p(0,1)^{p/2},
    \end{equation}
    for some constant $C = C(p) \geq 1$, which only depends on $p$.
    For a.e. $\rho \in (0,1)$, it holds that $u-a_\pm\in W^{1,p}(S_\pm)$, where $S_{\pm}=\{\pm
    x_2>\varepsilon\}\cap\dd B_\rho$, and $u - a_\pm$ has a tangential derivative which can be computed from
    the restriction of $\D u-A_\pm$. By Chebyshev's inequality and Fubini's theorem, one can further select a
    radius $\rho \in (1/2,1)$ such that
    \begin{equation}\label{eq_r2}
        \int_{S_\pm}|\D u-A_\pm|^p\,d\H^1\leq 4 \int_{D_\pm}|\D u-A_{\pm}|^p\,dx.
    \end{equation}
    We want to construct a competitor for $(u,K)$ in $B_\rho$. We set
    \begin{equation}
        G=(P \cap B_\rho) \cup Z \cup (K\setminus B_\rho), 
    \end{equation}
    where $Z=\{y\in\dd B_\rho \;:\; \abs{x_2} \leq \beta(0,1)\}$. It holds that
    \begin{equation}\label{eqq2}
        \mathcal{H}^1(G \cap \overline{B_\rho})\leq 2\rho + C\beta(0,{1}).
    \end{equation}
    We want to extend the functions $u_{|S_+}$ and $u_{|S_-}$ to the whole sphere $\dd B_\rho$. We reason with
    $u_{|S_+}$, the other being similar. Let $\phi$ be a diffeomorphism of $\dd B_\rho$ such that $\phi(S_+)=\dd
    B_\rho\cap\{x_2>0\}$. We extend $((u-a_+)\circ\phi^{-1})_{|\dd B_\rho\cap\{x_2>0\}}$ by symmetry to the whole circle $\dd B_\rho$, obtaining a function $\tilde{u}_+$. We set $u_+=\tilde{u}_+\circ\phi$. By construction,
    it holds that $u_+\in W^{1,p}(\dd B_\rho)$, that $u_+ = u - a_+$ on $S^+$ and that
    \begin{equation}\label{eq_r3}
        \int_{\dd B_\rho}|\D_\tau u_+|^p\,d\H^1\leq C\int_{S_+}|\D_\tau(u-a_+)|^p\,d\H^1\leq C\int_{S_+}|\D u-A_+|^p\,d\H^1,
    \end{equation}
    for some constant $C=C(p) \geq 1$, which only depends on $p$.
    By \cite[Lemma 22.32]{d}, \eqref{eq_r1}, \eqref{eq_r2}, \eqref{eq_r3}, there exist a function $v_+\in
    H^1(B_\rho)$ which extends $u_+$ on $B_\rho$ such that
    \begin{equation}\label{eqq3}
        \begin{aligned}
            \int_{B_\rho}|\D v_+|^2\,dx
             & \leq C \biggl(\int_{\dd B_\rho}|\D_\tau u_+|^p\,d\H^1\biggr)^{2/p}\\
             & \leq C \biggl(\int_{S_+}|\D u-A_+|^p\,d\H^1\biggr)^{2/p}\\
             & \leq C \biggl(\int_{D_+}|\D u-A_+|^p\,dx\biggr)^{2/p} \leq C \omega_p(0,1)
        \end{aligned}
    \end{equation}
    Similarly we construct $v_-\in H^1(B_\rho)$ such that 
    \begin{equation}
        \label{eqq4}
        \int_{B_\rho}|\D v_-|^2\,dx\leq C\omega_p(0,1).
    \end{equation}
    We then define the function $v$ as follows:
    \begin{equation}
        v=
        \begin{cases}
            v_- + a_-\quad&\text{in }B_\rho\cap  {\{x_2<0\}},\\
            v_+ + a_+ &\text{in }B_\rho\cap  {\{x_2>0\}},\\
            u & \text{in }\R^2\setminus (B_\rho\cup K).
        \end{cases}
    \end{equation}
    It is clear that $v_-+a_-$ and $u$ glue well along ${\dd B_\rho\cap\{x_2<0\}\setminus Z}$ and the same holds for $v_++a_+$ and $u$
    along ${\dd B_\rho\cap\{x_2>0\}\setminus Z}$, which implies that $v$ has weak derivatives in $\Omega \setminus G$. Furthermore,
    the pair $(v,G)$ is a topological competitor for $(u,K)$ in $B_\rho$ (see \cite[Proposition 45.1]{d}). By
    minimality, we have
    \begin{align*}
        \int_{B_\rho\setminus K} \A e(u):e(u) \,dx+\H^1(K\cap\overline{B_\rho})
& \leq \int_{B_\rho\setminus G} \A e(v):e(v) \,dx+\H^1(G\cap\overline{B_\rho})\\
& \leq C \int_{B_\rho}|\D v_-|^2\,dx+ C\int_{B_\rho}|\D v_+|^2\,dx\\
&\hphantom{\leq} + \H^1(G\cap\overline{B_\rho}),
    \end{align*}
    where we have also used that $e(a_-)=e(a_+)=0$. Combining the previous estimate, \eqref{eqq2}, \eqref{eqq3},
    \eqref{eqq4} and using our assumption $\beta(0,1) + \omega(0,1) \leq C \varepsilon$, we get
    \begin{equation}
        \H^1(K\cap\overline{B_\rho})-2\rho
        \leq C(\omega_p(0,1)+\beta(0,1))\leq C\varepsilon \leq 2 C \varepsilon \rho,
    \end{equation}
    which implies the thesis.
\end{proof}

An additional lemma shows that if in a ball centered on the boundary of a connected component the normalized
energy and the flatness are small, then this component must fill one side of the ball.  This prevents an
infiltration of a component of $\Omega \setminus K$ in a narrow strip.

\begin{lemma}[Infiltration Lemma] 
    \label{LemmaStrip}
    For every $p\in(1,2]$, there exists $\varepsilon_0=\varepsilon_0(p,\A)>0$ such that the following holds: if
    $(u,K)$ is a global minimizer of the Griffith functional and if $\Omega_0$ is a connected component of
    $\R^2\setminus K$, for every $\varepsilon\in(0,\varepsilon_0)$, $x_0\in\dd\Omega_0$ and $r>0$ such that
    \begin{equation}
        \omega_p(x_0,r)+\beta_K(x_0,r) \leq \varepsilon,
    \end{equation}
    then $\Omega_0$ contains one of the two connected components of 
    \[B(x_0,r)\setminus \{x\in B(x_0,r) \mid \dist(x,P) \leq \varepsilon r\},\]
    where $P$ is a line which achieves the infimum in the definition of the flatness.
\end{lemma}
\begin{proof}
    Without loss of generality, we may assume that $B(x_0,r) = B_1$.
    Fix any $\delta \in (0,1/10)$ and assume by contradiction that
    \begin{equation}
        \Omega_0\cap B_1 \subset \{x \in B_1 \mid \dist(x,P) \leq \varepsilon\},
    \end{equation}
    where $\varepsilon \leq \varepsilon_0 < \delta/2$, and $P$ is a line passing through the origin.
    To simplify the notation, we also assume that $P=\{x_2=0\}$. 
    Since
    $0\in\dd\Omega_0\subset\overline{\Omega_0}$, there exists a point $z_0 \in\Omega_0$ such that $\abs{z_0} <
    \delta$. By the unboundedness of $\Omega_0$, we have that $\Omega_0\cap\dd B_1 \neq\emptyset$. Furthermore,
    since $\Omega_0$ is a connected open set, $\Omega_0$ is arcwise connected. Therefore there exists a curve
    $\gamma$ in $\Omega_0$ that joins $z_0$ with a point $z = (z_1,z_2) \in \Omega_0 \cap \partial B_1$.
    Without loss of generality, we may assume that $z_1 \geq 0$, in particular such a point $z$ satisfies
    \begin{equation}
        z_1 \geq \abs{z} - \abs{z_2} \geq 1 - \varepsilon > 1 - \delta.
    \end{equation}
    Recalling that $\gamma$ starts at $z_0$ such that $\abs{z_0} \leq \delta$, we deduce that for all $t \in
    (\delta,1 - \delta)$, the vertical line $L_t := \{x_1 = t\}$ meets $\spt(\gamma) \cap B_1 \subset \Omega_0 \cap
    B_1$. It also meets both components of $\{x \in B_1 \mid \dist(x,P) > \varepsilon \}$, which are contained in
    $\Omega_0^c$, whence
    \begin{equation}
        \card{B_1 \cap \dd\Omega_0\cap L_t} \geq 2.
    \end{equation}
    Let $y = (y_1,y_2) \in B_1 \cap \dd \Omega_0$ be one of the points of $B_1 \cap \dd\Omega_0$ that the line
    $L_{1/2}$ meets. Notice that for all $t \in (1/4+\delta, 3/4 - \delta)$, any point $x = (x_1,x_2)$ in $B_1 \cap
    \dd \Omega_0 \cap L_t$ necessarily belongs to $B(y,1/4)$ since one has 
    \begin{equation}
        \abs{x_1 - y_1} = \Bigl\lvert x_1 - \frac{1}{2} \Bigr\rvert < \frac{1}{4} - \delta
        \quad \text{and} \quad
        \abs{x_2 - y_2} \leq \abs{x_2} + \abs{y_2} \leq 2\varepsilon < \delta. 
    \end{equation}
    Thus,
    \begin{equation}
        \card{\dd\Omega_0 \cap B(y,1/4) \cap L_t} \geq 2,\quad\forall t \in \Bigl(\frac{1}{4} + \delta, \frac{3}{4}- \delta\Bigr).
    \end{equation}
    On one hand, by the coarea formula we get
    \begin{equation}
        \label{eqeq27}
        \H^1\bigl(\dd \Omega_0 \cap B(y,1/4)\bigr)
        \geq \int_{1/4 +\delta}^{3/4 - \delta} \card{\dd\Omega_0 \cap B(y,1/4) \cap L_t}\,dt \geq 2\Bigl(\frac{1}{2}-2\delta\Bigr)=1-4\delta.
    \end{equation}
    On the other hand, by the scaling inequalities \eqref{eq_scaling}, we have
    \begin{equation}
        \omega_p(0,1/4) + \beta(0,1/4) \leq 4^{\frac{4}{p}-1}\varepsilon.
    \end{equation}
    and thus, assuming $\varepsilon_0$ small enough, Lemma~\ref{LemmaDensity} implies
    \begin{equation}
        \H^1\bigl(\dd \Omega_0\cap B(y,1/4)\bigr)\leq \H^1\bigl(K\cap B(y,1/4)\bigr)\leq \frac{1}{2}+ C4^{\frac{4}{p}-2}\varepsilon.
    \end{equation}
    Choosing finally $\varepsilon_0< 1/(100 C4^{4/p - 1})$, the previous chain of inequalities and \eqref{eqeq27} lead to a
    contradiction. We conclude that there exists a point of $\Omega_0$ which belongs to one of the two connected
    components of $\{x\in B_1 \;:\; \dist(x,P) > \varepsilon\}$. Since 
    \begin{equation}
        B_1 \cap \dd\Omega_0 \subset B_1 \cap K \subset \{x \in B_1 \mid \dist(x,P) \leq \varepsilon\},
    \end{equation}
    and $\Omega_0$ is connected, then we obtain the thesis.
\end{proof}

The last preliminary result follows from quantitative rectifiability properties of Griffith minimizers
\cite[Corollary 3.3]{camilleUR}.

\begin{lemma}\label{LemmaCarleson}
    For every $p\in(1,2]$, $\varepsilon>0$ and for every constant $C_0 \geq 1$ there exists a constant
    $a=a(p,\A,C_0,\varepsilon) \in (0,1/2)$ such that the following holds: if $(u,K)$ is a global minimizer of the Griffith
    functional, $x\in K$, $r>0$ and $E\subset K\cap B(x,r)$ is a measurable set such that $\H^1(E)\geq
    C_0^{-1}r$ then there exist $y\in E \cap B(x,r/2)$ and $t\in (ar,r/2)$ such that
    \begin{equation}
        \omega_p(y,t)+\beta(y,t) \leq \varepsilon.
    \end{equation}
\end{lemma}
\begin{proof}
    From \cite[Theorem 1.1]{camilleUR}, we know that $K$ is a uniformly rectifiable set. In particular by
    \cite[Corollary 41.5]{d} and \cite[Theorem 23.8]{d} we know that 
    \begin{equation}
        \big[\beta^2(y,t)+\omega_p(y,t)\big]\frac{d\H^1(y)dt}{t}
    \end{equation}
    is a Carleson measure on $K\times(0,+\infty)$, which means that for all $x \in K$ and $r > 0$,
    \begin{equation}\label{eq_carleson}
        \int_{K\cap B(x,r)}\int_0^r\big[\beta^2(y,t)+\omega_p(y,t)\big]\frac{d\H^1(y)dt}{t}\leq C_p r.
    \end{equation}

    Let $x \in K$, $r > 0$ and let $E \subset K \cap B(x,r)$ be a measurable set such that $\H^1(E) \geq C_0^{-1} r$.
    Fix a constant $a \in (0,1/2)$ and assume that for all $y \in E$ and all $t \in (ar,r/2)$,
    $\omega_p(y,t)+\beta(y,t)> \varepsilon$, in particular $\omega_p(y,t) + \beta(y,t)^2 \geq
    \varepsilon^2/4$. Using \eqref{eq_carleson}, this implies
    \begin{align*}
        C_p r &\geq \int_{K\cap B(x,r)}\int_0^r [\beta(y,t)^2+\omega_p(y,t)]\frac{d\H^1(y)dt}{t}\\
              &\geq \frac{\varepsilon^2}{4} \H^1(E) \int_{ar}^{r/2}\frac{1}{t}dt\geq \frac{\varepsilon^2}{ {4}}C_0^{-1}\log\bigg(\frac{1}{2a}\bigg)r,
    \end{align*}
    which leads to a contradiction for $a = \exp(-8C_p C_0 \varepsilon^{-2})/2$.
\end{proof}

We can now use the previous lemmata to establish the main ingredient for the existence of escape curves. The following
result states that every connected component of $\Omega \setminus K$ satisfies an interior corkscrew condition. While
this already prevents the existence of cusps, this property is not yet strong as the existence of escape paths. For instance, a
component $\Omega_0$ given as a dyadic chain of balls connected by narrow necks would satisfy the proposition below,
but would not admit escape paths. We will rule out this kind of counter-example in
Lemma~\ref{LemmaCurvaPreliminareLoc}.

\begin{proposition}[Corkscrew condition]\label{prop_bigdisk}
    \label{PropDisk}
    There exists $a = a(\A) >0$ such that if $(u,K)$ is a global minimizer of the Griffith functional, $\Omega_0$
    is a connected component of $\R^2\setminus K$, $x_0\in\Omega_0$ and $r>0$, then $\Omega_0\cap B(x_0,r)$
    contains a disk of radius $ar$.
\end{proposition}

\begin{proof}
    Let $\Omega_0$ be a connected component of $\R^2\setminus K$, let $x_0\in\Omega_0$ and $r>0$. If $B(x_0,r/4)
    \subset\Omega_0$, then we get the thesis for $a=1/4$. Otherwise, $B(x_0,r/4) \cap
    \partial \Omega_0\neq\emptyset$. We want to show that there exists a constant $C_0 \geq 1$ (which only depends on
    $\A$) such that
    \begin{equation}\label{eq_bigdisk_cr}
        \H^1\bigl(\dd\Omega_0\cap B(x_0,3r/4)\bigr)\geq C_0^{-1} r.
    \end{equation}
    If $\Omega_0$ is the only connected component of $\R^2\setminus K$ which intersects $B(x_0,r/2)$, then $K\cap
    B(x_0,r/2) \subset \dd\Omega_0$. Indeed, as $K$ is negligible for the Lebesgue measure, the set $\R^2 \setminus
    K$ is dense so any point of $K \cap B(x_0,r/2)$ can be approximated by points of $B(x_0,r/2) \setminus K$,
    which can only belong to $\Omega_0$.
    We conclude by lower Ahlfors-regularity of $K$, that
    \begin{equation}
        \H^1\bigl(\dd\Omega_0\cap B(x_0,3r/4)\bigr) \geq \H^1\bigl(K\cap B(x_0,r/2)\bigr)\geq Cr.
    \end{equation}

    Otherwise, $B(x_0,r/2)$ meets another connected component $\Omega_1\neq\Omega_0$. Let $x_1\in\Omega_1\cap
    B(x_0,r/2)$. Since $\Omega_1$ is unbounded, there exists a curve $\gamma_1$ connecting $x_1$ to $\dd
    B(x_0,r)$ such that $\spt(\gamma_1) \subset \Omega_1$. Analogously, we can find a curve $\gamma_0$ connecting $x_0$
    to $\dd B(x_0,r)$ such that $\spt(\gamma_0) \subset\Omega_0$. Since $\spt(\gamma_0)$ and $\spt(\gamma_1)$ meet $\partial
    B(x_0,\rho)$ for every $\rho\in (r/2, 3r/4)$, then 
    \begin{equation}
        \dd B(x_0,\rho) \cap \dd \Omega_0 \neq \emptyset,\quad \forall\rho\in (r/2,3r/4).
    \end{equation}
    This implies that, setting $\pi(z)=|z-x_0|$, we have that $\pi\bigl(\dd\Omega_0\cap B(x_0,3r/4) \bigr)$
    contains an interval of length $r/4$, and therefore
    \begin{equation}
        \H^1\bigl(\dd\Omega_0\cap B(x_0,3r/4)\bigr)\geq \H^1\bigl(\pi(\dd\Omega_0\cap B(x_0,3r/4))\bigr)\geq\frac{r}{4}.
    \end{equation}
    This achieves the proof of \eqref{eq_bigdisk_cr}.

    Let $\varepsilon>0$ and $p\in(1,2]$. Applying Lemma~\ref{LemmaCarleson} with $E := \dd \Omega_0 \cap
    B(x_0,3r/4)$, we deduce that there exist $b=b(p,\A,C_0)>0$, $y\in \dd\Omega_0\cap B(x_0,3r/4)$ and $t\in
    (br,r/2)$ such that
    \begin{equation}
        \omega_p(y,t)+\beta(y,t) \leq \varepsilon.
    \end{equation}
    Furthermore, by Lemma~\ref{LemmaStrip}, $\Omega_0$ contains one of the two connected components of
    $B(y,t)\setminus \{x\in B(y,t) \;:\; \dist(x,P(y,t))<\varepsilon t\}$. Thus, $\Omega_0$ contains a disk whose
    radius is comparable with $br$, which is the thesis.
\end{proof}


From the previous proposition, we show that any point in $\Omega \setminus K$ can be connected to another point twice
further away from $K$, while the path remains at controlled distance from the crack.

\begin{lemma}
    \label{LemmaCurvaPreliminareLoc}
    There exist two constants $\kappa \geq 1$ and $\tau>0$ (which only depend on $\A$) such that if $(u,K)$ is a
    topological almost-minimizer with gauge $h$ in $\Omega$, and if $x_0\in\Omega\setminus K$ and $r > 0$ are such
    that $\mathrm{dist}(x_0,K) \geq r$,
    \begin{equation}
        B(x_0,\kappa r)\subset\Omega\quad \text{and} \quad h(\kappa r) \leq \tau,
    \end{equation}
    then there exists a curve $\gamma$ joining $x_0$ to a point $z \in \Omega \setminus K$ such that 
    \begin{equation}\label{eq_curve}
        \dist(z,K)\geq 2r,\quad \spt(\gamma) \subset(\Omega\setminus K)\cap B(x_0,\kappa r/2),\quad \dist(\gamma,K)\geq \kappa^{-1}r.
    \end{equation}
\end{lemma}
\begin{proof}
    Without loss of generality, we may assume that $x_0=0$ and $r=1$. Let us suppose by contradiction that for
    any $k\in\N$, setting $\kappa=2^k$ and $\tau=2^{-k}$, there exist an open set $\Omega_k$ and a topological
    almost-minimizer $(u_k,K_k)$ in $\Omega_k$ with gauge $h_k$ such that $0\in\Omega_k\setminus K_k$,
    \begin{equation}
        \label{eqeq23}
        \dist(0,K_k) \geq 1,\quad B(0,2^k)\subset\Omega_k, \quad h_k(2^k)\leq 2^{-k},
    \end{equation}
    but a curve $\gamma$ as in the statement of the theorem does not exist. Since $B(0,2^k)\subset\Omega_k$,
    by Proposition~\ref{PropDisk} we have that $\Omega_k$ converges to $\R^2$ and thus, by \cite[Lemma 3.2]{lablem},
    up to subsequences the sequence $((u_k,K_k))_{k\in\N}$ converges to a limit $(u,K)$ with
    $\dist(0,K) = \lim_{k\to+\infty}\dist(0,K_k) \geq 1$. 
    Since $\lim_{k \to +\infty} h_k(r) = 0$ for all $r > 0$, an application of \cite[Theorem 2.7]{lablem} shows that
    $(u,K)$ is a topological almost-minimizer in $\R^2$ with gauge $h = 0$, that is, $(u,K)$ is a global
    minimizer in $\R^2$. Let $\Omega_0$ be the connected component of $\R^2\setminus K$ such that $0\in\Omega_0$. By
    Proposition~\ref{PropDisk}, there exists a ball $B(x_1,3)\subset\Omega_0$. Let $\gamma$ be a curve joining $0$
    and $x_1$ in $\Omega_0$. Since $\dist(x_1,K)>3$, since $K_k$ converges to $K$ and since $\spt(\gamma) \subset
    B(0,2\mathrm{diam}(\gamma))$, we have for $k$ sufficiently large
    \begin{equation}
        \dist(x_1,K_k)>2,\quad\dist(\gamma,K_k)>\frac{\dist(\gamma,K)}{2}>2^{-k},\quad \spt(\gamma) \subset B(0,2^{k-1}),
    \end{equation}
    which is a contradiction.
\end{proof}

\begin{remark}[Controlling the length of the curve]\label{rmk_curve}
    In the statement of Lemma~\ref{LemmaCurvaPreliminareLoc}, we only control the diameter of $\gamma$ and not
    its length. But one can actually replace $\gamma$ by a $\kappa$-Lipschitz map $\gamma : [0,r] \longrightarrow W$,
    where $W:=(\Omega\setminus K)\cap B(x_0,\kappa r/2)$,
    satisfying \eqref{eq_curve} with a potentially bigger constant $\kappa$.

    The standard construction, which we include for the reader's convenience, is as follows. Let us call by $\gamma$ a
    curve given by Lemma~\ref{LemmaCurvaPreliminareLoc} for some radius $r > 0$ such that $\mathrm{dist}(x_0,K)
    \geq r$, $B(x_0,\kappa r) \subset \Omega$ and $h(\kappa r) \leq \tau$. In the case that $\abs{z - x_0} < \kappa^{-1} r/2$, we can
    replace $\gamma$ by a direct parametrization of $[x_0,z]$; one checks easily that \eqref{eq_curve} is satisfied.
    Otherwise, we let $x_1$ be the \emph{last time}
    $\gamma$ leaves $B(x_0,\kappa^{-1} r/2)$, in particular $\abs{x_1 - x_0} = \kappa^{-1} r/2$ and
    $\mathrm{dist}(x_1,\R^2 \setminus W) \geq \kappa^{-1} r$.
    We continue the process to find points $x_0,\ldots,x_n$ such that 
    \begin{equation}
        \mathrm{dist}(x_i,\R^2 \setminus W) \geq \kappa^{-1} r,
        \quad
        \abs{x_i - x_{i-1}} = \kappa^{-1} r/2
    \end{equation}
    and
    \begin{equation}
        \abs{x_j - x_i} \geq \kappa^{-1} r /2 \quad \text{for all} \; j > i,
    \end{equation}
    until we arrive at a point $x_n$ such that $\abs{z - x_n} < \kappa^{-1} r/2$. Notice that since $\spt(\gamma)
    \subset B(x_0,\kappa r/2)$, all the points $x_i$ remain within $B(x_0,\kappa r/2)$. Given that for $i \ne j$,
    $\abs{x_j - x_i} \geq \kappa^{-1} r /2$, the balls $B(x_i,\kappa^{-1}r/4)$ are disjoint and, since they are
    contained in $B(x_0,\kappa r)$, we deduce that there are at most $C(\kappa)$ points $x_0,\ldots,x_n$, where
    $C(\kappa) \geq 1$ is a generic constant which depends on $\kappa$. 
    The chain of segments composed of $[x_i,x_{i+1}]$ for $i < n$ and of $[x_n,z]$ therefore has a total length $L$ such that
    $C(\kappa)^{-1} r \leq L \leq C(\kappa) r$. This allows to
    parametrize this chain by a $C(\kappa)$-Lipschitz map $\gamma_0 : [0,r] \longrightarrow \Omega
    \setminus K$. As $\mathrm{dist}(x_i,\R^N \setminus W) \geq \kappa^{-1} r$ and each segment is of length $\leq
    \kappa r/2$, the curve $\gamma_0$ stays at distance $\geq \kappa^{-1} r /2$ from $\R^N \setminus W$.
\end{remark}

Applying iteratively the previous Lemma, we obtain the existence of escape curves.

\begin{proposition}\label{PropJohnCurveLoc}
    There exist two constants $\alpha_J\in(0,1)$ and $\tau>0$ (which only depend on $\A$) such that if $(u,K)$ is
    a topological almost-minimizer in $\Omega$ with gauge $h$, $x_0\in\Omega\setminus K$, if $\ell > 0$ is such that
    \begin{equation}
        B(x_0,\ell) \subset \Omega \quad \text{and} \quad h(\ell) \leq \tau,
    \end{equation}
    then there exists a 1-Lipschitz curve $\gamma\colon[0,\ell]\longrightarrow\Omega\setminus K$ such that $\gamma(0)=x_0$ and
    \begin{equation}
        \label{eqeq24}
        \dist(\gamma(t),K)\geq \alpha_J t,\quad\forall t\in[0,\ell].
    \end{equation}
\end{proposition}
\begin{proof}
    Fix a point $x_0 \in \Omega \setminus K$ and radius $r = \mathrm{dist}(x_0,K)$. Let $n$ be the integer such that
    $(2^n - 1) \kappa r \leq \ell < (2^{n+1} - 1) \kappa r$.
    If $n = 0$, then $\ell \leq \kappa r$, that is,
    \begin{equation}
        \mathrm{dist}(x_0,K) \geq \kappa^{-1} \ell.
    \end{equation}
    It is then clear in this case that the constant path $\gamma \equiv x_0$ on $[0,\ell]$ satisfies \eqref{eqeq24}.

    Let us now assume that $n \geq 1$. We have in particular $\kappa r \leq \ell$ and we can thus apply
    Lemma~\ref{LemmaCurvaPreliminareLoc} and Remark~\ref{rmk_curve} to find a point $x_1 \in \Omega \setminus K$ such that $\mathrm{dist}(x_1,K) \geq 2r$ and
    a $\kappa$-Lipschitz map $\gamma_0 : [0,r] \longrightarrow \Omega \setminus K$ which connects $x_0$ to $x_1$ and stays at distance $\geq \kappa^{-1} r$ from $K$. 
    We iterate the process and obtain a finite sequence of points $(x_k)_k$ such that $x_k$ is at distance $\geq 2^k r$
    from $K$
    and $x_k$ is connected to $x_{k+1}$ by a $\kappa$-Lipschitz map $\gamma_k : [0,2^k r] \longrightarrow \Omega
    \setminus K$ such that $\mathrm{dist}(\gamma_k,K) \geq 2^k \kappa^{-1} r$.
    Notice that by construction $\abs{x_{k+1} - x_k} \leq 2^k \kappa r$ so that $\abs{x_k - x_0} \leq \sum_{i < k} 2^i
    \kappa r = (2^k - 1) \kappa
    r$. One can apply Lemma~\ref{LemmaCurvaPreliminareLoc} to connect $x_k$ to $x_{k+1}$ as long as
    \begin{equation}
        B(x_k,2^k \kappa r) \subset \Omega \quad \text{and} \quad h(2^k \kappa r) \leq \tau
    \end{equation}
    and, since $B(x_k,2^k \kappa r) \subset B(x_0,(2^{k+1} - 1) \kappa r)$, this is guaranteed whenever $(2^{k+1} - 1) \kappa r
    \leq \ell$; hence the process stops at $k = n$.

    For each $k = 0,\ldots,n-1$, we translate the domain of
    $\gamma_k$ to be $[(2^k - 1)r, (2^{k+1}-1)r]$ and we concatenate all the $\gamma_k$ to find a $\kappa$-Lipschitz map
    $\gamma : [0,(2^n - 1)r] \longrightarrow \Omega\setminus K$ connecting $x_0$ to $x_n$. Moreover, for all $t \in [0,(2^n - 1)r]$, there exists $k = 0,\ldots,n-1$ such that $t \in [(2^k-1)r,(2^{k+1}-1)r]$ so
    \begin{equation}
        \mathrm{dist}(\gamma(t),K) \geq 2^k \kappa^{-1} r \geq 2^k \kappa^{-1} \Bigl(\frac{t}{2^{k+1}-1}\Bigr) \geq
        \frac{\kappa^{-1} t}{2}.
    \end{equation}
    Observe that by definition of $n$, we have $\ell < 8(2^n - 1) \kappa r$. We can replace $\gamma$ by $\tilde \gamma : t
    \mapsto \gamma(\kappa^{-1}t/8)$, in order to obtain a $1$-Lipschitz map defined on $[0,8(2^n - 1)\kappa r]$, and in
    particular on $[0,\ell]$ by restriction; this reparametrization preserves the John condition with only a worse constant.
\end{proof} 

\section{Morrey-Campanato type estimates for the elastic energy}
\label{section5}

We prove Theorem~\ref{main0}, stated in introduction (see Theorem~\ref{holderestim} below). This establishes in
particular a Hölder-type estimate, which requires the development of Morrey-Campanato theory for the symmetric
gradient.


\subsection{Oscillation control with the symmetric gradient}
\label{SubsectionOscillation}
\newcommand{\Ce}{C_*}

Throughout this subsection, we assume that $\Omega \subset \R^2$ is an open set, $u \in H^1_{loc}(\Omega;\R^2)$ is a
Sobolev vector field, and there exists $\Ce>0$ such that for all $B(x,r)\subset \Omega$ it holds 
\begin{equation}
    \label{eqeq17}
    \int_{B(x,r)} |e(u)|^2 \, dx \leq \Ce r.
\end{equation}
We will prove that, up to removing a suitable global rigid movement in $\Omega$, the displacement $u$ satisfies
Hölder estimates, thus providing an analog of Morrey-Campanato theory for the symmetric gradient.

We begin with the following lemma, which controls the difference between two infinitesimal rotations in nearby balls of
comparable radii.


\begin{lemma}\label{LemmaMovRig}
    Let $\Omega\subset\R^2$ be an open set and let $u\in H^1_{\mathrm{loc}}(\Omega;\R^2)$ be a map satisfying \eqref{eqeq17}.
    Let $B(x_1,r_1),B(x_2,r_2)\subset\Omega$ with $r_1,r_2>0$ and assume that there exists a constant $M>0$ and a ball $V\subset B(x_1,r_1)\cap B(x_2,r_2)$ such that $B(x_1,r_1)\cup B(x_2,r_2)\subset MV$.
    Then
    \begin{equation}
        |A_{B(x_1,r_1)}-A_{B(x_2,r_2)}|\leq C \max(r_1,r_2)^{-1/2},
    \end{equation}
    for some constant $C=C(\Ce,M) \geq 1$.
\end{lemma}
\begin{proof}
    Let $r$ denote the radius of $V$. Observe that as $V\subset B(x_1,r_1)\cap B(x_2,r_2)$ and $B(x_1,r_1)\cup B(x_2,r_2)\subset MV$, we have
    \begin{equation}
        r \leq \min(r_1,r_2) \leq \max(r_1,r_2) \leq M r,
    \end{equation}
    so all the radii $r, r_1, r_2$ are comparable, up to a multiplicative constant which depends only on $M$.
    In the following, $C$ denotes a generic constant which may depend on $\Ce$ and $M$.
    Using the triangle inequality, the inclusion $V \subset B(x_1,r_1) \cap B(x_2,r_2)$, Korn's inequality and \eqref{eqeq17}, we obtain
    \begin{align*}
        |A_{B(x_1,r_1)}-A_{B(x_2,r_2)}|^2 &= \fint_{V} |A_{B(x_1,r_1)}-A_{B(x_2,r_2)}|^2 \, dx\\
                                          &\leq C \fint_{V} |A_{B(x_1,r_1)}-\nabla u|^2 \, dx + C \fint_{V} |A_{B(x_2,r_2)}-\nabla u|^2 \, dx\\
                                          &\leq C \fint_{B(x_1,r_1)} |A_{B(x_1,r_1)}-\nabla u|^2 \, dx + C \fint_{B(x_2,r_2)} |A_{B(x_2,r_2)}-\nabla u|^2 \, dx\\
                                          &\leq C \fint_{B(x_1,r_1)} \abs{e(u)}^2 \, dx + C \fint_{B(x_2,r_2)} \abs{e(u)}^2 \, dx\\
                                          &\leq C (r_1^{-1} + r_2^{-1}) \leq C \max(r_1,r_2)^{-1}.
    \end{align*}
\end{proof}

We will also require the case of two concentric balls with possibly non-comparable radii. This follows from an iterative
application of the preceding argument.

\begin{lemma}\label{LemmaMovRig2}
    Let $\Omega\subset\R^2$ be an open set and let $u\in H^1_{\mathrm{loc}}(\Omega;\R^2)$ be a map satisfying \eqref{eqeq17}.
    Let $B(x,r)\subset\Omega$ and let  $\rho <r$.    Then
    \begin{equation}
        |A_{B(x,\rho )}-A_{B(x,r)}|\leq C \rho^{-1/2},
    \end{equation}
    for some constant $C$ depending only on $\Ce$.
\end{lemma}
\begin{proof} Let $r_k := 2^{-k}r$ for $k\geq 0$ and let $A_k:=A_{B(x,r_k)}$. We estimate 
    \begin{align*}
        |A_{k}-A_{k+1}|^2 &= \fint_{B_{k+1}} |A_{k}-A_{k+1}|^2 \, dy\\
                          &\leq C  \fint_{B_{k+1}} |A_{B_k}-\nabla u|^2 \, dy + C \fint_{B_{k+1}} |A_{B_{k+1}}-\nabla u|^2 \, dy\\
                          &\leq C  \fint_{B_{k}} |A_{B_k}-\nabla u|^2 \, dy + C \fint_{B_{k+1}} |A_{B_{k+1}}-\nabla u|^2 \, dy\\
                          &\leq C \fint_{B_k} \abs{e(u)}^2 \, dy + C \fint_{B_{k+1}} \abs{e(u)}^2 \, dy\\
                          &\leq C  r_k^{-1}.
    \end{align*}
    If $k_0$ is chosen such that $r_{k_0+1}< \rho \leq r_{k_0}$ we also have, using the same argument,  that 
    \[|A_{k_0}-A_{B(x,\rho)}|\leq Cr_{k_0}^{-\frac{1}{2}}\leq C \rho^{-\frac{1}{2}}.\]
    Finally,
    \[|A_{B(x,\rho)} - A_{B(x,r)}|\leq |A_{B(x,\rho)} - A_{k_0}| + \sum_{k=0}^{k_0-1} |A_k - A_{k+1}| \leq C
    \rho^{-\frac{1}{2}},\]
    which finishes the proof.
\end{proof}

Then, we show that in a John domain where \eqref{eqeq17} holds, a single infinitesimal rotation in $\Omega$ controls
all other infinitesimal rotations, in any arbitrary ball.

\begin{lemma}
    \label{LemmaMovRig3}
    Let $\Omega \subset \R^2$ be a John domain with constant $\alpha_J$ and center $x_0$.
    Let $u \in H^1_{\mathrm{loc}}(\Omega;\R^2)$ be a map satisfying \eqref{eqeq17}.
    Let $r_\Omega:=\dist(x_0,\Omega^c)$.  Then, for all $x\in \Omega$ and  $\rho\geq 0$ such that $B(x,\rho)\subset\Omega$, we have
    \begin{equation}
        |A_{B(x,\rho)} - A_{B(x_0, r_\Omega)}|\leq C\rho^{-\frac{1}{2}}, \label{mainESTIM1}
    \end{equation}
    where   $C=C(\Ce,\alpha_J)$.
\end{lemma}
\begin{proof}
    Let    $x\in \Omega$ and  $\rho  \geq 0$ be such that $B(x,2\rho)\subset\Omega$ and let $\gamma : [0,\ell]
    \longrightarrow \Omega$ be a John curve from $x$ to $x_0$. Let $(D_i)_{i\in\N}$ be the chain of balls with centers $x_i$ and radii $r_i$ defined in Proposition~\ref{Chain}.
    Assume first that $\rho \in (0,\ell)$. Then there exists a unique $k\in\N^*$  such that $\rho\in[q^k\ell,q^{k-1}\ell)$.
    By \eqref{eqeq13} we have
    \begin{equation}
        |x_k-x|\leq q^k\ell\leq\rho,
    \end{equation}
    so that $x_k \in B(x,\rho)$.
    Since $r_k$ and $\rho$ are comparable, reasoning as in Proposition~\ref{Chain}, it is possible to find a ball $V\subset D_k \cap B(x,\rho)$ such that $D_k \cap B(x,\rho)\subset MV\subset \Omega$.
    Thus, applying Lemma~\ref{LemmaMovRig} and recalling that $(r_i)_{i\in\N}$ is a geometric sequence of ratio $q \in (0,1)$, we infer that
    \begin{align}
        |A_{B(x,\rho)} - A_{D_0}| &\leq |A_{B(x,\rho)}-A_{D_k}|+\sum_{i=0}^{k-1} |A_{D_{i+1}}-A_{D_i}|\\
                                  &\leq C \rho^{-\frac{1}{2}} + C \sum_{i=0}^{k-1} r_i^{-1/2}\\
                                  &\leq C \rho^{-\frac{1}{2}} + C r_k^{-1/2} \leq C \rho^{-1/2}.
    \end{align}
    Here the constant depends on $\Ce$ and $M$, and thus on $\Ce$ and $\alpha_J$.  Moreover, 
    \[r_\Omega=\dist(x_0,\Omega^c) =\dist(\gamma(\ell),\Omega^c)\geq \alpha_J \ell,\]
    from which we deduce that $r_0:=\alpha_J \ell /2 \leq  r_\Omega $. Applying Lemma \ref{LemmaMovRig} we get
    \[|A_{D_0} - A_{B(x_0,r_\Omega)} | \leq C r_0^{-\frac{1}{2}}\leq C\rho^{-\frac{1}{2}},\]
    so that finally 
    \[|A_{B(x,\rho)} - A_{B(x_0,r_\Omega)}| \leq |A_{B(x,\rho)} - A_{D_0}| +|A_{D_0} - A_{B(x_0,r_\Omega)} | \leq
    C\rho^{-\frac{1}{2}},\]
    which proves the lemma in the case when $\rho \in (0,\ell)$.

    It remains to treat the case when  $\rho \geq \ell$.  First assume that $B(x,2\rho)\subset \Omega$. We notice that  since $\gamma$ is $1$-Lipschitz,
    \[|x-x_0|=|\gamma(0)-\gamma(\ell)|\leq \ell \leq \rho.\]
 This means that $x_0 \in B(x,\rho)$, from which we directly get the  existence of  a ball of radius $\rho/4$ in $B(x,\rho)\cap B(x_0,\rho)$. On the other hand  $B(x,\rho)\cup B(x_0,\rho)$ is contained in $B(x,2\rho)\subset \Omega$ thus  Lemma \ref{LemmaMovRig} applies with $M=8$ and  says  that 
    \[|A_{B(x,\rho)}-A_{B(x_0,\rho)}|\leq C\rho^{-\frac{1}{2}},\]
    where $C$ depends only on $C^*$ and $\alpha_J$.  Moreover, since $B(x_0,\rho) \subset B(x,2\rho)\subset \Omega$, we  must have  $\rho \leq r_\Omega$.  Then we can use Lemma \ref{LemmaMovRig2} which yields
      \[|A_{B(x_0,\rho)}-A_{B(x_0,r_\Omega)}|\leq C\rho^{-\frac{1}{2}},\]
    which all together proves that 
    \[|A_{B(x,\rho)}-A_{B(x_0,r_\Omega)}|\leq C\rho^{-\frac{1}{2}},\]
    which finishes the proof in the case $\rho \geq \ell$ and when $B(x,2\rho)\subset \Omega$. Finally if $B(x,2\rho)$ is not contained in $\Omega$ we can apply all the preceding to the ball $B(x,\rho/2)$, yielding 
    \[|A_{B(x,\rho/2)}-A_{B(x_0,r_\Omega)}|\leq C\rho^{-\frac{1}{2}},\]
  and then use Lemma   \ref{LemmaMovRig2} to control    $|A_{B(x,\rho/2)}-A_{B(x,\rho)}|$ and conclude with the triangle inequality.
\end{proof}

As a consequence of the above result we can prove a Morrey type control on the full gradient for  $u-A_{B(x_0,r_\Omega)}$. Here is the statement.
 
\begin{lemma}
    \label{morraycontrol1}
    Let $\Omega \subset \R^2$ be a John domain with constant $\alpha_J$ and center $x_0$, and let $u \in H^1_{\mathrm{loc}}(\Omega;\R^2)$ be a map satisfying \eqref{eqeq17}.    Define  $r_\Omega:= \dist(x_0,\Omega^c)$ and consider the function  $v$ defined by    
    \[v(x):=u(x)-A_{B(x_0,r_\Omega)} x, \quad  \forall x \in \Omega.\] 
    Then, for all $r  \geq 0$ and $x\in \Omega$ such that $B(x,r)\subset\Omega$, we have
    \[\int_{B(x,r)}|\nabla v|^2 \;dy \leq C r,\]
    with $C=C( C^*,\alpha_J)$.
\end{lemma}

\begin{proof} Let $x\in \Omega$ and $r\geq 0$ be such that $B(x,r)\subset \Omega$. Then Lemma \ref{LemmaMovRig3} applies and setting $A_0:=A_{B(x_0, r_\Omega)}$ we get
    \[ |A_{B(x,r)} - A_{0}|\leq Cr^{-\frac{1}{2}}.\]
    But then using Korn inequality \eqref{Korn2}, together with the previous inequality we can write
    \begin{eqnarray}
        \int_{B(x,r)}|\nabla v|^2 \; dy &=&\int_{B(x,r)}|\nabla u -A_{0}|^2 \; dy  \notag \\
                                        &\leq & 2 \int_{B(x,r)}|\nabla u -A_{B(x,r)}|^2 \; dy +2 \int_{B(y,r)}| A_{B(x,r)}-A_{0}|^2 \; dy  \notag \\
                                        &\leq &  C \int_{B(x,r)}|e(u)|^2 \; dx +  C r^{-1} |B(x,r)| \notag \\
                                        &\leq & C r,
    \end{eqnarray}
    where $C=C( C^*,\alpha_J)$. 
\end{proof}

With Lemma \ref{morraycontrol1} at hand, we can derive an oscillation control on the function $v$ using the standard
Morrey-Campanato theory. For instance we can state the following result, which, while not directly needed in this paper,
is interesting on its own.

\begin{corollary}\label{cor_morrey}
    Let $\Omega \subset \R^2$ be a John domain with center $x_0$, and let $u \in H^1_{\mathrm{loc}}(\Omega;\R^2)$ be a map satisfying \eqref{eqeq17}.    Define  $r_\Omega:= \dist(x_0,\Omega^c)$ and consider the function  $v$ defined by    
    \[v(x):=u(x)-A_{B(x_0,r_\Omega)} x, \quad  \forall x \in \Omega.\] 
    Then, $v\in C^{0,1/2}_{loc}(\Omega)$. Moreover if $\Omega$ is Lipschitz then $v\in C^{0,1/2}(\overline{\Omega})$.
\end{corollary} 
\begin{proof} This is a direct consequence of Lemma \ref{morraycontrol1} together with \cite[page 43]{giaquinta}.
\end{proof}

\begin{remark}
    Here, we focus on the special case of the Morrey-Campanato spaces $L^{2,3}$ in dimension 2 with $p=2$,
    though the same proofs extend to the general spaces $L^{p,\lambda}$ as in \cite{giaquinta} and to higher dimensions.
\end{remark}

From Corollary \ref{cor_morrey}, we now can identify every map $u \in H^1_{\mathrm{loc}}(\Omega;\R^2)$ satisfying
\eqref{eqeq17} with its continuous representative. In the sequel we will use a more precise oscillation estimate,
which we state and prove below.

\begin{lemma}
    \label{LemmaTwoCurves}
    Let $\Omega\subset\R^2$ be a John domain with constant $\alpha_J$ and center $x_0$, and let
    $r_\Omega:=\dist(x_0,\Omega^c)$. Let $u\in H^1_{\mathrm{loc}}(\Omega;\R^2)$ be a map satisfying \eqref{eqeq17}.
    Let $x,y\in\Omega$ and $\gamma_1 : [0,\ell_1]\longrightarrow\Omega$ and $\gamma_2 : [0,\ell_2]\longrightarrow\Omega$ be two John curves respectively from $x$ to $x_0$ and from $y$ to $x_0$.
    Then, setting $v(x)=u(x)-A_{B(x_0,r_\Omega)}x$, for $x\in\Omega$, we have
    \begin{equation}
        \label{eqeq39}
        |v(x)-v(y)|\leq C\ell^{1/2},
    \end{equation}
    where $C=C(\Ce, \alpha_J) \geq 1$ and $\ell=\max\{\ell_1,\ell_2\}$.  In particular,
    \begin{equation}
        |u(x)-u(y)|\leq C\ell^{1/2}(1+|A_{B(x_0,r_\Omega)}|\ell^{1/2}). \label{Holderiann}
    \end{equation}
\end{lemma}

\begin{proof} Let $(D_i)_{i\in\N}$ be the chain of balls with centers $x_i$ and radii $r_i$ converging to $x$ defined in Proposition~\ref{Chain}.
    We recall that $r_i=q^i r_0$, for some $q\in(0,1)$ and for every $i\in\N$, where $r_0=\alpha_J \ell_1/4$. We also recall the existence of a ball $V_i\subset D_i\cap D_{i+1}$ such that $D_i\cup D_{i+1}\subset M V_i\subset \Omega$.
       In the following, $C$ denotes a generic constant which may depend on $\Ce$, $M$, $q$, $\alpha_J$ (and in turn, only on $\Ce$ and $\alpha_J$).

    To prove  \eqref{eqeq39} let us first show that
    \begin{equation}
        |v(x)-m_{D_0}|\leq C\ell_1^{1/2}. \label{step1toprove}
    \end{equation}
    We recall that two subsequent balls $D_i$, $D_{i+1}$ have comparable radii $C^{-1} r_{i+1} \leq r_i \leq C r_{i+1}$ and relatively closed centers $\abs{x_i - x_{i+1}} \leq C r_i$.
    Recall also that by construction of the balls $D_i$, we have $\abs{x - x_i} \leq C r_i$, for all $i\in\N$.
    As $v$ is continuous (even locally Hölder), it is clear that $\displaystyle{\lim_{i \to + \infty}} m_{D_i} = v(x)$.
    Then, we infer by a telescopic argument
    \begin{equation}
        \label{eqeq88}
        |v(x)-m_{D_0}| \leq \sum_{i=0}^{+\infty} |m_{D_{i+1}} - m_{D_i}|.
    \end{equation}
    Recall that by Poincar\'e  inequality, for any ball $D\subset \Omega$,
    \begin{equation}
        \int_{D} \abs{v(z) - m_{D} }^2 \, dz \leq C |D| \int_{D} \abs{\nabla v}^2 \, dz.
    \end{equation}
    Moreover, we recall that $D_i \cup D_{i+1} \subset MV_i \subset \Omega$.

       This means that Lemma \ref{morraycontrol1} applies and says that for all $i$,
    \[\int_{M V_i}|\nabla v|^2\, dx\leq C r_i.\]

    Now we can estimate,
    \begin{eqnarray}
        |m_{D_i}-m_{MV_i}|&=& \left|\fint_{D_i} v - m_{MV_i} \;dz\right|\leq C \fint_{MV_i}|v-m_{MV_i}|  dz  \notag \\
                          &\leq&  C r_i^{-1}  \left(\int_{MV_i}|v-m_{MV_i}|^2  dz \right)^{\frac{1}{2}}
                          \leq  C\left(  \int_{MV_i}|\nabla v|^2  dz \right)^{\frac{1}{2}} \notag \\
                          &\leq & Cr_i^{1/2}.
    \end{eqnarray}
    Here $C$ depends on $M$, and thus on $\alpha_J$. A similar reasoning also yields 
    \[|m_{D_{i+1}}-m_{MV_i}|\leq  Cr_k^{\frac{1}{2}},\] 
    so that from  the triangle inequality we get $  \abs{m_{D_i} - m_{D_{i+1}}} \leq C r_i^{\frac{1}{2}}$,
    whence, by properties of geometric series, we conclude from \eqref{eqeq88} that
    \begin{align}
        |v(x)-m_{D_0}|  \leq C \sum_{i=0}^{+\infty} r_i^{1/2}  \leq C r_0^{1/2}  \leq C \ell_1^{1/2},
    \end{align}
    and \eqref{step1toprove} follows.

    \medskip

    Next by use of a very similar argument we also establish   the following estimate:
    \begin{equation}
        \label{eqeq30}
        |m_{D_0}-v(x_0)|\leq C\ell_1^{1/2}.
    \end{equation}
    Indeed, for every $i\in\N$, let us consider $C_i=B(x_0,\rho_i)$, where $\rho_i=2^{-i}r_0$. We observe that $C_0=D_0$.  Reasoning as before we easily prove that 
    \[    \abs{m_{C_i} - m_{C_{i+1}}}\leq C\rho_i^{1/2}.\]
    As $v$ is continuous, we directly have that $m_{C_i} \xrightarrow[i \to +\infty]{} v(x_0)$.
    Then, reasoning as in step 1, we obtain the estimate \eqref{eqeq30}.


    \medskip

    All in all, we have proved that 
    \[|v(x)-v(x_0)|\leq C \ell_1^{1/2}.\]
    Arguing exactly in the same way for $y$, we also have 
    \begin{equation}
        \label{eqeq37}
        |v(y)-v(x_0)|\leq C\ell_2^{1/2},
    \end{equation}
    and  \eqref{eqeq39} follows from the triangle inequality.

    Now to prove \eqref{Holderiann}, denoting $A_0:=A_{B(x_0,r_\Omega)}$ and using the fact that $|x-y|\leq 2\ell$ (recall
    $\gamma_i$ is $1$-Lipschitz),
    \begin{equation}
        |u(x)-u(y)|\leq |v(x)-v(y)|+|A_0||x-y|\leq C\ell^{1/2}(1+|A_0|\ell^{1/2}),
    \end{equation}
    which concludes the proof of the lemma.
\end{proof}

\subsection{H\"older bounds far from the singular set}
\label{SubsectionFarSing}

As an application of the oscillation estimate on $u$ obtained in the preceding section, we first derive an H\"older-type
estimates for $u$ away from the singular set. In the absence of a gauge $h$, this would alternatively follow from
elliptic regularity.

\begin{proposition}\label{reginterieure}
    Let $(u,K)$ be a topological almost-minimizer for the Griffith energy on $\Omega\subset \R^2$. Then for any ball $B(x_0,r) \subset \Omega \setminus K$ such that $h(r) < +\infty$, we have $u\in H^1(B(x_0,r))\cap L^\infty(B(x_0,r))$.
\end{proposition}
\begin{proof}
    The fact that $u\in H^1(B(x_0,r))$ directly follows from the Korn inequalities \eqref{Korn1} and
    \eqref{Korn2}. Then we notice that the gradient control \eqref{eqeq17} holds true in all balls $B' \subset B(x_0,r)$ (with a constant $\Ce$ which depends only on $\A$ and $h(r)$).
    The ball $B(x_0,r)$ is clearly John domain with center $x_0$: every $x \in B(x_0,r)$ can be connected to $x_0$ by an escape path of length $\leq r$ and a universal John constant $\alpha_J$.
    Therefore, Lemma~\ref{LemmaTwoCurves} can be applied in $B(x_0,r)$ implying that there exists a skew-symmetric matrix $A$ such that for all $x \in B(x_0,r)$,
    \begin{equation*}
        \abs{u(x) - u(x_0)} \leq C r^{1/2} (1 + r^{1/2} \abs{A}).
    \end{equation*}
    This proves that $u \in L^{\infty}(B(x_0,r))$.
\end{proof}

\subsection{Local John domains and H\"older bounds near the singular set}
\label{SubsectionNearSing}

We now prove the first part of the main results of  this paper, stated as follows.

\begin{theorem}\label{holderestim}
    There exists $\kappa \geq 3$, $\tau > 0$, $n_0 \in \N$ (which depend only on $\A$) such that the following holds.
    Let $(u,K)$ be a topological almost-minimizer with gauge $h$ in $\Omega$.
    Let $x_0 \in K$ and $r_0 > 0$ be such that
    \begin{equation*}
        B(x_0,\kappa r_0) \subset \Omega \quad \text{and} \quad h(\kappa r_0) \leq \tau.
    \end{equation*}
    Then   for all $r \in (0,r_0]$,
    there exists a finite covering $(U_i)_i$ with at most $n_0$ connected open subsets of $B(x_0,3r) \setminus K$ such that
    \begin{equation}
        B(x_0,r)\setminus K\subset \bigcup_i U_i\subset B(x_0,3r)\setminus K
    \end{equation}
    and each $U_i$ is a John domain with center $x_i\in \partial B(x_0,2r)$ and with a John constant depending only on $\A$. Moreover, 
    \begin{equation}
        |u(x)-u(y)|\leq Cr^{1/2},\quad \text{for all} \; x,y\in U_i,\quad \text{and for all} \; i, \label{estiMM}
    \end{equation}
where the constant $C>0$  may depend on  $r_0$, $\alpha_J$, $\bf A$, and on  $u$ itself, but does not depend on $r$.\end{theorem}


\begin{proof}[Proof of Theorem \ref{holderestim}]
    Let $x_0 \in K$. Without loss of generality, we may assume that $x_0 = 0$.
    Let $\alpha_J \in (0,1)$ and $\tau > 0$ denote the constants appearing in Proposition~\ref{PropJohnCurveLoc}.
    Fix $\kappa := 1+2\alpha_J^{-1} > 3$ and $r_0 > 0$ such that
    \begin{equation}
        B(0, \kappa r_0) \subset \Omega \quad \text{and} \quad h(\kappa r_0) \leq \min(1,\tau).
    \end{equation}
    Note that \eqref{eqeq17} holds true in all balls $B' \subset B(x_0,\kappa r_0)$ with a constant $\Ce$ which depends
    only on $\A$.
    Thus we will be in position to apply Lemma~\ref{LemmaTwoCurves} and Lemma~\ref{LemmaMovRig2}.
    Let $r \in (0,r_0]$.

    We first check that for all $x \in B_r \setminus K$,
    there exists a $1$-Lipschitz curve $\gamma\colon[0,2r]\longrightarrow\Omega\setminus K$
    such that $\gamma(0)=x$ and 
    \begin{equation*}
        \spt(\gamma) \cap\dd B_{2r}\cap \{\dist(\cdot,K)\geq 2\delta r\}\neq\emptyset,
    \end{equation*}
    where $\delta := \alpha_J/4$, and
    \begin{equation}
        \label{eqeq35}
        \dist(\gamma(t),K)\geq\alpha_J t,\quad\forall t\in[0,\ell],
    \end{equation}
    where $\alpha_J > 0$ depends only on $\A$. 
    By applying Proposition~\ref{PropJohnCurveLoc} with $\ell = 2 \alpha_J^{-1} r$, we obtain an escape curve $\gamma : [0,\ell] \longrightarrow \Omega \setminus K$ starting from $x$.
    Notice that since $0 \in K$,
    \begin{equation}
        \abs{\gamma(\ell)} = \abs{\gamma(\ell) - 0} \geq \mathrm{dist}(\gamma(\ell),K) \geq \alpha_J \ell = 2r.
    \end{equation}
    We deduce that there exists $s \in [0,\ell]$ such that $\gamma(s) \in \partial B_{2r}$. Moreover, using   $\gamma(s) \in \dd B_{2r}$, $\gamma(0) \in B_r$ and the fact that $\gamma$ is $1$-Lipschitz, we have
    $r \leq \abs{\gamma(s) - \gamma(0)} \leq s$ so $r \leq s$. Then,
    \begin{equation}
        \mathrm{dist}(\gamma(s),K) \geq \alpha_J s \geq \alpha_J r \geq 2 \delta r,
    \end{equation}
    by definition of $\delta$. This proves our claim.

    Next, consider a maximal family of points $(x_i)_{1 \leq i \leq  n(r)}$ in $\dd B_{2r}\cap \{\dist(\cdot,K)\geq 2\delta r\}$ such that $\abs{x_i - x_j} \geq \delta r$.
    It is clear that for all $i$, $B(x_i,\delta r) \subset \{\dist(\cdot,K) \geq \delta r\}$ and that the set \[E := \{x_i \mid 1 \leq i \leq n(r)\}\] is $2\delta r$-dense in $\dd B_{2r} \cap \{\dist(\cdot,K) \geq 2 \delta r\}$.
    Moreover, we note that since the balls $B(x_i,\delta r/2)$ are mutually disjoint and contained in $B_{3r}$, there must be at most $n(r) \leq n_0$ such points, where $n_0 \in \N$ depends only on $\delta$, and thus only on $\A$.

    Now we prove that every point of $B_r\setminus K$ can be joined to a point of $E$ by means of a curve which satisfies the escape condition. Recall that for all $x \in B_r \setminus K$, there exists a $1$-Lipschitz curve $\gamma \colon [0,\ell] \longrightarrow \Omega \setminus K$ such that $\gamma(0) = x$ and such that there exists $s \in [r,\ell]$ satisfying $\gamma(s) \in \dd B_{2r} \cap \{\dist(\cdot,K)\geq 2\delta r\}$.  We may require that 
    \begin{equation}
        \label{eqeq28}
        s=\min \big\{t\in[r,\ell]\,:\,\gamma(t)\in\dd B_{2r}\cap \{\dist(\cdot,K)\geq 2\delta r\}\big\}.
    \end{equation}

    Since $\gamma(s) \in \dd B_{2r} \cap \{\dist(\cdot,K) \geq 2 \delta r\}$, there exists $x_i \in E$ such that $\abs{x_i - \gamma(s)} < 2\delta r = \alpha_J r/2$.
    Then by applying Lemma~\ref{LemmaJohnSegment} to the restriction of $\gamma : [0,s] \longrightarrow \Omega \setminus K$, there exists a $1$-Lipschitz curve $\tilde{\gamma} : [0,2s] \longrightarrow \Omega \setminus K$ such that $\tilde{\gamma}(0) = x$, $\tilde{\gamma}(2s) = x_i$ and
    \begin{equation}\label{eqeq14}
        \dist(\tilde{\gamma}(t),K)\geq \alpha_J t/4, \quad\forall t\in[0,2s].
    \end{equation}
By the construction of $\tilde{\gamma}$ and \eqref{eqeq28}, we have that $\mathrm{spt}(\tilde{\gamma})\subset \overline{B}_{2r}$.


    We define
    \begin{align}
        \Gamma_i=\{
    & x\in B_r\setminus K\,:\,\text{there exists $\ell > 0$ and a 1-Lipschitz curve }\gamma\colon[0,\ell]\rightarrow \overline{B}_{2r}\setminus K\\
    &\text{s.t.}\,\gamma(0)=x,\,\gamma(\ell)=x_i, \gamma\text{ satisfies }\eqref{eqeq14}\}.
    \end{align}
    and then  
    \begin{equation}
        U_i = \bigcup_{x\in \Gamma_i}B(x,\dist(x,K)). \label{defUi}
    \end{equation}
    Note that some $U_i$'s may be empty. It is immediate to show that
    \begin{equation}
        B_r\setminus K\subset\bigcup_{i=1}^{n(r)}U_i\subset B_{3r}\setminus K.
    \end{equation}

    Fix $i$ such that $U_i \ne \emptyset$.   We define $c := \alpha_J/4 < 1$, the constant appearing in \eqref{eqeq14}. Let us prove that $U_i$ is a John domain, with center $x_i$ and constant $c/10$. It is clear that $U_i$ is open, by the definition. Let now $y\in U_i$, and let us distinguish two cases. In the first case we assume that 
    $y \in \Gamma_i$. 
    Then there exists an escape path $\gamma:[0,\ell] \longrightarrow \R^2$ from $y$ to $x_i$, with constant $c$. We claim that this escape path is a John curve for $U_i$. Indeed, we know that for all $t\in [0,\ell]$,
    \[\dist(\gamma(t),K)\geq c  t.\]
    Now let $z=\gamma(t_0)$ for some fixed $t_0 \in [0,\ell]$. Then the curve $t \mapsto \gamma(t+t_0)$ defined on $[0,\ell - t_0]$ is still an escape path starting at $z$, up to $x_i$, with constant $c$. This shows that $z \in \Gamma_i$. Then, by definition of $U_i$, the ball $B(z,\dist(z,K))$ is contained in $U_i$ as well so that
    \[\dist(\gamma(t_0),U_i^c)\geq \dist(\gamma(t_0),K)\geq c  t_0,\]
    which proves that the curve $\gamma$ is a John curve for $U_i$ with constant $c$, starting at $y$.

    Now in a second case we assume that $y \in U_i \setminus \Gamma_i$. Then it means that there exists $y'\in \Gamma_i$ such that $y\in B(y', \dist(y',K))$. By definition of $\Gamma_i$, there exist a $c$-escape path $\gamma$ from $y'$ to $x_i$. As shown in Lemma \ref{LemmaJohnSegment}, the curve obtained by adding the segment $[y,y']$ to the support of $\gamma$ can be parameterized by an escape path $\tilde \gamma$ on $[0,\ell+t_0]$, where $t_0=|y-y'|$, with constant $c/10$. 
    We already know, by the reasoning used in the first case, that for all 
    $t\in [t_0,\ell + t_0]$, the whole ball $B(\tilde\gamma(t),\dist(\tilde \gamma(t),K))$ is contained in $U_i$ so that 
    \[\dist(\tilde \gamma(t),U_i^c) \geq \dist(\tilde \gamma(t),K) \geq \frac{ct}{10},\]
    where the last inequality comes from the fact that $\tilde \gamma$ is an escape curve of constant $c/10$.
    Then on the remaining part $[0,t_0]$, the curve $\tilde \gamma$ consists of the arc-length parameterisation of the segment 
    $[y,y']$, in other words $\tilde \gamma(t)=y+ t(y'-y)/t_0$. Moreover   we know that $B:=B(y', \rho)$, with $\rho:=\dist(y',K)$, is contained in $U_i$ because $y' \in \Gamma_i$. Recall that by definition, $y \in B(y',\rho)$ so $t_0 \leq \rho$. This implies that for $t \in [0,t_0]$,
    \begin{align}
        \dist(\tilde \gamma(t), U_i^c)\geq \dist(\tilde \gamma(t), B^c) &= \rho - \abs{\tilde \gamma(t) - y'}\\
                                                                        &= \rho-(t_0 - t) \geq t\geq \frac{ct}{10},
    \end{align}
    where we have used that   $c\leq 1$. This proves that $\tilde \gamma$ is a John curve associated to $U_i$, from $y$ to $x_i$ and   achieves the proof of the fact that $U_i$ is a John domain with center $x_i$, and constant $c/10$.

    By Lemma~\ref{LemmaTwoCurves}, and recalling that the length of a John curve in $U_i$ is always bounded by $C \mathrm{diam}(U_i) \leq Cr$
    by Remark \ref{boundl}, we have that for all $x,y \in U_i$,
    \begin{equation}\label{eqeq19}
        |u(x)-u(y)|\leq C r^{1/2}(1+r^{1/2} |A_{B(x_i, r_i)}|),
    \end{equation}
    with $r_i:=\dist(x_i,U_i^c)$. 
    \end{proof}

\subsection{Finite number of traces}
\label{SubsectionTraces}

Finally, we give the proof of Theorem~\ref{main3}, which is a direct consequence of Theorem~\ref{holderestim}.

\begin{proof}[Proof of Corollary~\ref{main3}]
    Let $n_0$ be the constant of Theorem~\ref{holderestim}. Assume for a contradiction that $E(x_0)$ contains more than $n_0$ points. Let $(a_k)_{1\leq k\leq n_0+1}$ be a set of $n_0+1$ points in $E(x_0)$ and set $\delta:= \min_{k\neq k'}|a_k-a_{k'}|$ be the minimal distance between them.
    Thanks to Theorem~\ref{holderestim}, we can find a radius $r_0 > 0$ such that for any $r\in (0,r_0)$ there exists a covering $(U_i)_{1\leq i\leq n_0}$ of connected subsets of $B(x_0,3r)\setminus K$ (some of the $U_i$'s may be possibly empty) such that
    \begin{equation}
        |u(x)-u(y)|\leq C_0r^{1/2},\quad\forall x,y\in U_i. \label{holderEstim00}
    \end{equation}
    We can additionally choose $r$ so small that
    \begin{equation}
        C_0 r^{1/2}\leq \frac{\delta}{4},
    \end{equation}
    This radius $r$ is now fixed for the rest of the proof.

    Now let $a \in E(x_0)$. This means that there exists a sequence $(y_k)_k$ in $\Omega \setminus K$ such that $y_k\to x_0$ and $u(y_k)\to a$.
    By the finiteness of the $U_i$'s we can further extract a subsequence (not relabeled) such that the sequence $(y_k)_{k\in\N}$ lies in a single set $U_i$. 
    Then, consider an other sequence $(z_k)_{k\in\N}$ converging to $x_0$ in the same set $U_i$, with $u(z_k)\to a'$. By \eqref{holderEstim00}, it holds
    \[|u(y_k)-u(z_{k})|\leq C_0 r^{1/2}\leq \frac{\delta}{4}\]
    and passing to the limit, we obtain
    \begin{equation}
        |a-a'|\leq \frac{\delta}{4}.\label{distll}
    \end{equation}
    Hence, all the possible limit values $a$ coming from $U_i$ are contained in a ball $B_i\subset \R^2$ of radius   $\delta/2$. 
    We conclude that
    \[E(x_0)\subset \bigcup_{i=1}^{n_0} B_i,\]
    with $B_i$ of radius  $\delta/2$. This contradicts the fact that $E(x_0)$ contains $n_0+1$ distinct points of mutual distance greater than $\delta$. We therefore conclude that $E(x_0)$ contains at most $n_0$ points.
\end{proof} 

\section{Proof of Theorem \ref{main1}}\label{SBV}

Nexte, provide a proof of Theorem~\ref{main1}, which is a consequence of Theorem~\ref{holderestim}.

\begin{proof}[Proof of   Theorem~\ref{main1}]
    Let $\kappa > 3$, $\tau > 0$ and $n_0 \in \N$ be the constants of Theorem~\ref{holderestim}. Let $(u,K)$ be an almost-minimizer for the Griffith energy in $\Omega \subset \R^2$.

    Let $x_0\in K$ be given. Applying Theorem~\ref{holderestim}, we get the existence of $r_0>0$ such that $B(x_0,\kappa r_0)\subset \Omega$ and a finite covering $(U_i)_{1\leq i\leq n_0}$  of $B(x_0,r_0)\setminus K$ by open subsets of $B(x_0,3r_0)\setminus K$ such that $U_i$ is a John domain with a constant depending only on $\A$.
    In virtue of Proposition~\ref{kornJohn} we know that for all $i$ we have a rigid motion $a_i$ such that 
    \[\int_{U_i} |u-a_i|^2\,dx \leq C {\rm diam}(U_i)^p \int_{U_i} |e(u)|^2 \,dx,\]
    \[\int_{U_i} |\nabla u -A_i|^2\,dx \leq C \int_{U_i} |e(u)|^2 \,dx,\]
    where $C \geq 1$ is a constant that depends only on $\A$.  
    Notice that since $e(u)\in L^2(U_i)$ for all $i$,  the estimates above imply that $u \in H^1(B(x_0,r_0)\setminus K)$. We also know from Theorem~\ref{holderestim} that for all $i$ and for all $x,y \in U_i$ we have
    \[|u(x)-u(y)|\leq C r_0^{1/2},\]
    for some constant $C \geq 1$. Fixing any arbitrary $x_i \in U_i$, the above estimate gives 
    \[\sup_{x\in U_i}|u(x)|\leq Cr_0^{1/2}+ |u(x_i)|,\]
    which implies $u\in L^\infty(U_i)$. Since there are at most $n_0$ domains $U_i$, we deduce that $u \in L^\infty(B(x_0,r_0)\setminus K)$. 
    This proves that for every $x \in K$ there exists a radius $r > 0$ such that
    $u \in H^1(B(x,r) \setminus K) \cap L^\infty(B(x,r))$.

    Now, we justify that if $U$ is an open set such that $\overline{U}\subset \Omega$, then $u \in H^1(U \setminus K) \cap L^{\infty}(U)$. Since we already know that $u \in H^1_{\mathrm{loc}}(\Omega \setminus K)$ by definition of almost-minimizers, it is only left to check that
    \begin{equation}\label{eq_sumH1L}
        \lVert \nabla u \rVert_{L^2(U)} + \lVert u \rVert_{L^\infty(U)} < +\infty.
    \end{equation}
    The set $K \cap \overline{U}$ is compact so it can be covered by finitely many balls $B_i = B(x_i,r_i)$, with $x_i \in K \cap\overline{U}$, $r_i > 0$, $B(x_i,r_i) \subset \Omega$ and \[u \in H^1(B_i \setminus K) \cap L^{\infty}(B_i).\]
    The set $\overline{U} \setminus \bigcup_i B(x_i,r_i)$ is also compact and is disjoint from $K$ so it can be covered by finitely many balls $B_j = B(y_j,r_j)$, with $y_j \in \overline{U}$, $r_j > 0$, $B(y_j,r_j) \subset \Omega \setminus K$ and $h(r_j) < +\infty$. In particular, Proposition~\ref{reginterieure} shows that \[u \in H^1(B_j) \cap L^{\infty}(B_j).\]
    As the domain $U$ is covered by the combined finite family $(B(x_k,r_k))_k \cup (B(y_k,r_k))_k$, we conclude that (\ref{eq_sumH1L}) holds.
    By \cite[Proposition~4.4]{afp}, this implies $u \in SBV(U)$, and since $\nabla u\in L^2(U\setminus K)$, we actually get $u \in SBV^2(U)$, which completes the proof of Theorem~\ref{main1}.
\end{proof}

\bibliographystyle{alpha}
\bibliography{biblio}

\medskip
\noindent
\textsc{Université de Lorraine, CNRS, IECL, F-54000 Nancy, France}

\end{document}